\documentclass[12pt]{amsart}
\usepackage{amssymb, latexsym, amsthm, enumitem, hyperref}

\usepackage{color}

\newtheorem{thm}{Theorem}[section] 
\newtheorem{assum}[thm]{Assumption}
\newtheorem{cor}[thm]{Corollary}
\newtheorem{defn}[thm]{Definition}
\newtheorem{exmpl}[thm]{Example}
\newtheorem{lem}[thm]{Lemma}
\newtheorem{prop}[thm]{Proposition}
\newtheorem{rem}[thm]{Remark}
\newtheorem{ques}[thm]{Question}
\newcommand\Cref[1]{{Corollary~\ref{#1}}}
\newcommand\Dref[1]{{Definition~\ref{#1}}}
\newcommand\Lref[1]{{Lemma~\ref{#1}}}
\newcommand\Pref[1]{{Proposition~\ref{#1}}}
\newcommand\Tref[1]{{Theorem~\ref{#1}}}
\newcommand\Sref[1]{{Section~\ref{#1}}}
\newcommand\Srefs[2]{{Sections~\ref{#1} and~\ref{#2}}}
\newcommand\Ssref[1]{{Subsection~\ref{#1}}}
\newcommand\Aref[1]{{Appendix~\ref{#1}}}

\renewcommand{\Pr}{\operatorname{P}}
\newcommand{\N}{\mathbb{N}}
\newcommand{\Z}{\mathbb{Z}}
\newcommand{\R}{\mathbb{R}}
\newcommand{\set}[1]{\left\{#1\right\}}

\newcommand{\E}{\mathbb{E}}

\newcommand{\ind}[1]{1_{#1}}
\newcommand{\Bin}{\mathrm{Bin}}
\newcommand{\Geo}{\mathrm{Geo}}
\newcommand{\range}{\mathrm{range}}
\newcommand{\diam}{\mathrm{diam}}
\newcommand{\eps}{\varepsilon}
\newcommand{\floor}[1]{\left\lfloor #1\right\rfloor}
\newcommand{\ceil}[1]{\left\lceil #1\right\rceil}
\newcommand{\dd}{\,\mathrm{d}}
\newcommand{\suchthat}{\,\ifnum\currentgrouptype=16\middle\fi|\,}

\title[A Law of Iterated Logarithm on Diagonal Products]{A Law of Iterated Logarithm on Lamplighter Diagonal Products}
\author{Gideon Amir}
\address{Department of Mathematics, Bar-Ilan University}
\email{gidi.amir@gmail.com}
\author{Guy Blachar}
\address{Department of Mathematics, Bar-Ilan University}
\email{guy.blachar@gmail.com}


\makeatletter
\@namedef{subjclassname@2020}{2020 Mathematics Subject Classification}
\makeatother

\begin{document}

\begin{abstract}
    We prove a Law of Iterated Logarithm for random walks on a family of diagonal products constructed by Brieussel and Zheng \cite{brzh}. This provides a wide variety of new examples of Law of Iterated Logarithm behaviours for random walks on groups. In particular, it follows that for any $\frac{1}{2}\leq \beta\leq 1$ there is a group $G$ and random walk $W_n$ on $G$ with  $\E|W_n|\simeq n^\beta$ such that
    $$0<\limsup \frac{|W_n|}{n^\beta(\log\log n)^{1-\beta}}<\infty$$
    and
    $$0<\liminf \frac{|W_n|(\log\log n)^{1-\beta}}{n^\beta}<\infty.$$
\end{abstract}

\maketitle

\section{Introduction}

Let $G$ be some finitely generated group, with a finite symmetric generating set $S$. Let $W_n$ be a random walk on $G$, with finitely supported symmetric step measure such that the support of~$W_1$ generates $G$. A central object of study in the theory of random walks on groups is the distribution of the distance $|W_n|$ of the random walk from its origin point, and its connection to other geometric and algebraic properties of $G$. One is interested both in understanding the behaviour of $|W_n|$ for specific families of  groups and in understanding the set of possible behaviours of the distance function for random walks on groups in general (often referred to as the ``inverse problem'').  Usually, fine understanding of the underlying metric structure of $G$ is lacking, and beyond some examples of polynomial growth groups and some non-amenable groups, most works focused on understanding the expected distance on some families of groups, with few works also looking at some moderate and large deviation regimes (See subsections \ref{ss:expected} and \ref{ss:LIL}).

The main contribution of this manuscript is proving a Law of Iterated Logarithm type result for random walks on a family of diagonal products studied in \cite{brzh}. These groups were used to capture a variety of behaviours of random walks on groups, in particular in terms of the expected distance. Our result provides a wide variety of Law of Iterated Logarithm type behaviours for random walks on groups.

\subsection{The expected distance $\E|W_n|$}\label{ss:expected}
Much advancement has been done in the last years regarding the expected value $\E|W_n|$, also known as the \textbf{speed} or \textbf{rate of escape} of the random walk. It is trivial to see that the expected distance is sub-additive, and in particular that $\E|W_n|\lesssim n$ (Where $\lesssim , \simeq$ denote (in)equalities up to an absolute constant that may depend on the choice of group and random walk measure but not on $n$). By \cite{leeper}, for any random walk $W_n$ one has $\E|W_n|\gtrsim \sqrt{n}$. For many ``small'' groups, such as virtually nilpotent groups, the random walk is always diffusive, that is $\E |W_n|\simeq \sqrt{n}$, while by a theorem of Kaimanovich and Vershik (see \cite{KV}) $\E|W_n|\simeq n$ for non-Liouville measures. For a long time all known examples of random walks on groups exhibited one of the two extreme behaviours above.

This lead to the following natural question, attributed to Vershik.
\begin{ques}
What functions can be realized as the speed function of a random walk on some group?
\end{ques}

The first ``non-classical'' examples towards this question were given by Erschler \cite{ersc}, who proved that if a random walk on a group $G$ satisfies $\E\left|W_n\right|\simeq n^{\alpha}$, then the random walk on $G\wr\Z$ satisfies $\E\left|W_n\right|\simeq n^{\frac{1+\alpha}{2}}$. Starting with $G=\Z$ and using this result recursively, one can realize all of the functions $f(n)=n^{1-2^{-k}}$ as speed functions of iterated wreath product.

The next step was taken by Amir and Vir\'{a}g \cite{amirvi}, who constructed groups where $W_n\simeq n^{\beta}$ for $\frac{3}{4}\leq \beta<1$ (and more generally with $\E|W_n|\simeq f(n)$ for any ``nice enough'' function $f$ within that range). Their construction uses permutational wreath product over the natural action of the mother groups~$\mathcal{M}_m$ on the boundary of their tree with $\Z$-valued lamps.

In \cite{brzh} Brieussel and Zheng constructed random walks on groups that, up to some regularity condition, capture the whole range of possible behaviours of the expected distance:
\begin{thm}[{\cite[Theorem 1.1]{brzh}}]\label{thm:brzh-main-thm}
Let $f\colon[1,\infty)\to[1,\infty)$ be a continuous function with $f(1)=1$ such that $\frac{f(x)}{\sqrt{x}}$ and $\frac{x}{f(x)}$ are non-decreasing. Then there exists a group $\Delta$ such that the random walk on $\Delta$ satisfies $\E\left|W_n\right|\simeq f(n)$.
\end{thm}

The examples constructed and analyzed in Brieussel and Zheng also gave new behaviours of other geometric quantities such as the entropy of the random walk, the return probabilities and the isoperimetric profile.

\subsection{The Law of Iterated Logarithm}\label{ss:LIL}
Since the possible behaviour of the expected distance is well understood in many cases, one may try and study further questions regarding the behaviour of the distance function of the random walk $W_n$. One such direction is to study the times in which the random walk is atypically far from its origin. This can be formulated by the Law of Iterated Logarithm.

The classical Law of Iterated Logarithm on $\Z$ originates from the works of Khintchine \cite{KH24} and Kolmogorov \cite{KO29}, and states that for any random walk $W_n$ with zero mean and unit variance satisfies
$$\limsup_{n\to\infty}\frac{|W_n|}{\sqrt{2n\log\log n}}=1$$
almost surely.

Unlike the expected distance which was well studied, there are much fewer examples where the Law of Iterated Logarithm of random walks on groups is understood. There are several ways to phrase a Law of Iterated Logarithm for a random walk on a general group. Perhaps the strongest one is finding a function $g(n)$ such that
$$\limsup_{n\to\infty}\frac{|W_n|}{g(n)}=1$$
almost surely. However, finding such a function requires a tight estimation for the distance $|W_n|$, which is in many cases hard to achieve.

We say that $g(n)$ is an \textbf{upper scaling function} for $W_n$, if
$$0<\limsup_{n\to\infty}\frac{|W_n|}{g(n)}<\infty$$
almost surely, and that $h(n)$ is a \textbf{lower scaling function} for $W_n$, if
$$0<\liminf_{n\to\infty}\frac{|W_n|}{h(n)}<\infty$$
almost surely.

These definitions can also be rephrased in terms of inner and outer radii. We say that a function $R(n)$ is an \textbf{outer radius} for $W_n$ if $|W_n|\ge R(n)$ finitely often with probability $1$, and a function $r(n)$ is an \textbf{inner radius} for $W_n$ if $|W_n|\leq r(n)$ finitely often with probability $1$. It follows that $g(n)$ is an upper scaling function if $Cg(n)$ is an outer radius for some constant $C>0$, but $cg(n)$ is not an outer radius for some other constant $c>0$ (and one can similarly rephrase lower scaling functions in terms of inner radii).

Let us briefly review previous works in this direction. The classical results for $\Z$ and $\Z^d$ can be found e.g.\ in \cite{Dv51}, where Dvoretzky and Erd\"{o}s give a characterization for the inner and outer radii of a simple random walk on $\Z^d$. Hebisch and Saloff-Coste generalize this theorem for arbitrary groups with polynomial volume growth of order $d\ge 3$, see \cite[Theorem 9.2]{He93}.

In \cite{Re03}, Revelle studies groups of the form $G\wr\Z$, where the random walk on $G$ has $\alpha$-tight degree of escape, and proves that $G\wr\Z$ has $\alpha'=\frac{1+\alpha}{2}$-tight degree of escape; he also shows that $n^{\alpha'}\left(\log\log n\right)^{1-\alpha'}$ is an upper scaling function for $G\wr\Z$, whereas $n^{\alpha'}/\left(\log\log n\right)^{1-\alpha'}$ is a lower scaling function for this group. Revelle studies in addition several Baumstag-Soliter groups, proving they have an inner radius of order $\sqrt{n/\log\log n}$ and an outer radius of order $\sqrt{n\log\log n}$.

Finally, in \cite{Th13} Thompson proves Laws of Iterated Logarithm for certain polycyclic and metaabelian groups. He shows that these groups are all diffusive, have $\frac{1}{2}$-tight degree of escape (which is a form of control over the tail of the distance function), and have an upper scaling function $g(n)=\sqrt{n\log\log n}$.

\subsection{Main results}
In this article, we consider the groups constructed in \cite{brzh}. These groups are diagonal products of lamplighter groups with finite lamp groups. We focus on the case where the lamp groups are expanders, and prove a Law of Iterated Logarithm on these groups. Our main theorem is the following:

\begin{thm}\label{thm:main-thm}
Let $f\colon[1,\infty)\to[1,\infty)$ be a continuous function such that $f(1)=1$ and such that $\frac{x}{f(x)}$ and $\frac{f(x)}{\sqrt{x}}$ are non-decreasing. Let $\Delta$ be the group from \cite{brzh} (in the expander case) for which the speed function is equivalent to $f(n)$. Write $W_n$ for the random walk on $\Delta$ with the appropriate generators.
\begin{enumerate}
    \item \label{eq:limsup main} If $\frac{f(x)}{\sqrt{x}\log\log\log x}$ is non-decreasing, then
    $$0<\limsup_{n\to\infty}\frac{|W_n|}{\log\log n\, f\left(\frac{n}{\log\log n}\right)}<\infty$$
    almost surely.
    \item \label{eq:liminf main} If $\frac{f(x)}{\sqrt{x}(\log\log x)^{1+\eps}}$ is non-decreasing for some $\eps>0$, then
    $$0<\liminf_{n\to\infty}\frac{|W_n|}{\frac{1}{\log\log n}f\left(n\log\log n\right)}<\infty$$
    almost surely.
\end{enumerate}
\end{thm}

\begin{exmpl}
Let $\frac{1}{2}<\alpha<1$, and suppose that the speed function we choose is $f(n)=n^\alpha$. Then the theorem states that
$$\Pr\left(0<\limsup_{n\to\infty}\frac{\left|W_n\right|}{n^\alpha (\log\log n)^{1-\alpha}}<\infty\right)=1$$
and
$$\Pr\left(0<\liminf_{n\to\infty}\frac{\left|W_n\right|(\log\log n)^{1-\alpha}}{n^\alpha}<\infty\right)=1.$$
\end{exmpl}

\begin{rem}
One could try to construct examples of groups satisfying a Law of Iterated Logarithm as in the example above by considering the wreath products $G \wr \Z$ where $G$ is taken to be a group with speed $n^\beta$, and then using Revelle's results \cite{rev}. However, there are two caveats to this approach. First, the wreath product  $G \wr \Z$ will have a rate of escape $n^{(\beta+1)/2}$, thus it could only provide examples with $\alpha\geq \frac{3}{4}$.   Second, in order to apply Revelle's theorems, one must know that the random walk on $G$ satisfies a \emph{tight degree of escape} which amounts to proving tail bounds on the distance of the random walk on $G$, which requires the same kind of analysis done in our paper.
\end{rem}

\begin{rem}
Throughout this article, we did not optimize the constants. Any unnumbered constant is some universal constant, but its value may change between claims. However, numbered constants keep their values for the rest of the paper.
\end{rem}

\section{Realizing Speed Functions with Diagonal Products}\label{sec:group-desc}

In this section we describe the groups from \cite{brzh}, and give an outline of the technique used in this article to estimate the speed of the random walk on these groups, proving \Tref{thm:brzh-main-thm}. These groups are diagonal products of a sequence of lamplighter groups with finite lamp groups, where the lamp groups are chosen to be expanders or diffusive groups. In this article we focus on the expander case, although we first give the general description for the groups.

\subsection{Diagonal product}

\begin{defn}
Let $X=\set{x_1,\dots,x_{|X|}}$ be a set, and let $\set{\Delta_s}_{s\ge 0}$ be a sequence of groups. Suppose that each $\Delta_s$ is generated by a set $X(s)$ which we identify as a copy of $X$ in $\Delta_s$, i.e.\ $X(s)=\set{x_1(s),\dots,x_{|X|}(s)}$. The \textbf{diagonal product of $\set{\Delta_s}_{s\ge 0}$ with respect to $\set{X(s)}_{s\ge 0}$}, which will be denoted $\Delta$, is the subgroup of the direct product $\prod_{s\ge 0}\Delta_s$ generated by the diagonal elements $(x_i(s))_{s\ge 0}$.

Alternatively, we can construct $\Delta$ as follows: let $F(X)$ be the free group generated by $X$, and let $\pi_s\colon F(X)\to\Delta_s$ denote the natural projections. Then $\Delta=F(X)/\bigcap_{s\ge 0}\ker\pi_s$.
\end{defn}

Note that $\Delta$ has a natural generating set, given by the diagonal elements.

\begin{rem}
If $X$ has a natural algebraic structure, for example a union of several groups, it is natural to require that the above construction will be compatible with that structure. This can be done by requiring that the identifications of $X(s)$ with $X$ respect that structure.
\end{rem}

\subsection{The lamp groups}

Let $A=\set{a_1,\dots,a_{|A|}}$ and $B=\set{b_1,\dots,b_{|B|}}$ be two groups. Let $\set{\Gamma_s}_{s\ge 0}$ be a sequence of groups, where each $\Gamma_s$ is generated by a set of the form $A(s)\cup B(s)$, where $A(s)$ and $B(s)$ are subgroups of $\Gamma_s$ isomorphic to $A$ and $B$ respectively.

Fix a sequence of strictly increasing integers $\set{k_s}$. For each $s$ let $\Delta_s=\Gamma_s\wr\Z$, with a generating set given by $\tau(s)=(e,+1)$, $\alpha_i(s)=(a_i(s)\delta_0,0)$ for all $1\leq i\leq |A|$, and $\beta_j(s)=(b_j\delta_{k_s},0)$ for all $1\leq j\leq |B|$.

Finally, we take $\Delta$ to be the diagonal product of the groups $\Delta_s$ with respect to the above generating sets, marked with the generating set $\mathcal{T}=(\tau,\alpha_1,\dots,\alpha_{|A|},\beta_1,\dots,\beta_{|B|})$. If $U_\alpha$ and $U_\beta$ are the uniform measures on the subgroups $\mathcal{A}=\set{\alpha_1,\dots,\alpha_{|A|}}$ and $\mathcal{B}=\set{\beta_1,\dots,\beta_{|B|}}$, and $\mu$ is the uniform measure on $\set{\tau,\tau^{-1}}$, we use the ``switch-walk-switch'' measure for our random walk on $\Delta$, that is
$$q=(U_\alpha*U_\beta)*\mu*(U_\alpha*U_\beta).$$

Denote by $W_n$ the random walk on $\Delta$ with step distribution $q$. By the choice of $q$, we may write $W_n=(f_s^{W_n},S_n)$, where $S_n$ is a simple random walk on $\Z$, and $f_s^{W_n}$ denotes the lamp configuration at the layer $\Delta_s$ at time $n$.

Following \cite{brzh}, we make the following assumptions on the groups $\Gamma_s$:
\begin{assum}~
\begin{enumerate}
    \item $k_0=0$ and $\Gamma_0=A(0)\times B(0)\cong A\times B$;
    \item Let $[A(s),B(s)]^{\Gamma_s}$ denote the normal closure of the subgroup generated by commutators $[a_i(s),b_j(s)]$. Then we assume
    $$\Gamma_s/[A(s),B(s)]^{\Gamma_s}\cong A(s)\times B(s)\cong A\times B.$$
    \item Letting $l_s=\diam(\Gamma_s)$, we assume that $k_s$ and $l_s$ grow at least exponentially.
\end{enumerate}
\end{assum}

In \cite{brzh} the authors treat two cases of families of groups $\Gamma_s$: one of them is when the random walks on $\Gamma_s$  are diffusive, and the other one is when the groups $\set{\Gamma_s}$ satisfy the following linear speed assumption:
\begin{defn}[{\cite[Definition 3.1]{brzh}}]
Let $\set{\Gamma_s}$ be a sequence of finite groups where each $\Gamma_s$ is marked with a generating set $A(s)\cup B(s)$. Let $\eta_s=U_{A(s)}*U_{B(s)}*U_{A(s)}$ where $U_{A(s)},U_{B(s)}$ are uniform distribution on $A(s),B(s)$. We say $\set{\Gamma_s}$ satisfies the \textbf{$(\sigma,T_{s})$-linear speed assumption} if in each $\Gamma_s$,
$$L_{\eta_s}(t)=\E\left|X_t^{(s)}\right|_{\Gamma_s}\ge\sigma t\ \mbox{for all }t\leq T_s$$
where $X_t^{(s)}$ has distribution $\eta_s^{*t}.$

Note that for the definition to be meaningful, we must have $\sigma\leq 1$.
\end{defn}

In this paper, we focus only on the case where the groups satisfy the linear speed assumption. Such groups can be constructed for instance using Lafforgue super expanders or with lamplighter groups over a $d$-dimensional infinite dihedral group, see \cite[Examples 3.2 and 3.3]{brzh}.

\subsection{Excursions and the speed in a single layer}\label{subs:speed-single-layer}

In a single layer $\Delta_s=\Gamma_s\wr\Z$, the lamp value at $x\in\Z$ can only be changed when the random walk reaches $x$ or $x-k_s$. As $A(s)$ is a subgroup and the measure chosen on it is uniform, a product of random elements from $A(s)$ is again a uniform element from~$A(s)$ (and the same holds for~$B(s)$). Therefore, conditioning on the trajectory of the simple random walk~$S_t$, the distribution of the lamp value at $x\in\Z$ depends only on the number of $k_s$-excursions:
\begin{defn}
Let $S_t$ be a random walk on $\Z$, and let $k>0$. We say that the random walk begins a \textbf{$k$-excursion} in time $j$, if it visits $S_j-k$ before its next visit in $S_j$. We denote by $T(k,x,n)$ the number of $k$-excursions the random walk $S_t$ performs before time $n$ starting at~$x$.
\end{defn}

To estimate the speed of the random walk on $\Delta_s$, Brieussel and Zheng give bounds on the speed in terms of the number of $k_s$-excursions and the diameter $l_s=\diam(\Gamma_s)$. By dividing the range into intervals of length $k_s$ and working in each one separately one after the other, one can show that for some universal constant $C>0$,
\begin{equation}\label{eq:brzh-eq1}
    \left|W_n\right|_{\Delta_s}\leq\frac{Ck_s}{2}\sum_{j\in\Z}T\left(\frac{k_s}{2},\frac{k_s}{2}j,n\right)+C\left|\range(S_n)\right|,
\end{equation}
which has an expected value of $\frac{C'n}{k_s}+C'\sqrt{n}\lesssim\frac{n}{k_s}$ if $k_s\leq\sqrt{n}$. Also, since $l_s=\diam(\Gamma_s)$, we also have the bound
$$\left|W_n\right|_{\Delta_s}\leq\left|\range(S_n)\right|l_s,$$
which has an expected value $\lesssim\sqrt{n}\,l_s$. Together we get
$$\E\left|W_n\right|_{\Delta_s}\lesssim\sqrt{n}\,l_s+\frac{n}{k_s}.$$
Brieussel and Zheng show that this also gives a matching lower bound because of the linear speed assumption, so one gets the asymptotic behavior of the speed of the random walk on $\Delta_s$.

\subsection{Estimating the speed of the random walk}

The random walk on $\Delta$ can be thought of as parallel random walks on each $\Delta_s$, so in order to estimate the speed in $\Delta$ one needs to use the estimates for the speed in each layer $\Delta_s$. In \cite{brzh}, Brieussel and Zheng show that there is some universal constant $C>0$ such that
\begin{equation}\label{eq:brzh-eq0}
    \left|W_n\right|_{\Delta_s}\leq\left|W_n\right|_{\Delta}\leq C\sum_{s:k_s\leq\left|\range(S_n)\right|}\left|W_n\right|_{\Delta_s}.
\end{equation}
As the speed in each layer $\Delta_s$ is of order $\sqrt{n}\,l_s+\frac{n}{k_s}$ and the sequences $\set{k_s}$ and $l_s$ grow at least exponentially, it follows that for each $n$ there is a layer $s_1(n)=\max\set{s\suchthat k_sl_s\leq\sqrt{n}}$ that determines the speed of the random walk, i.e.\ $\E[|W_n|_{\Delta}]$ is equivalent to $\E[|W_n|_{\Delta_{s_1(n)}}]$ up to universal constants. This proves the following:
\begin{prop}[{\cite[Proposition 3.6]{brzh}}]
$$\E[|W_n|_{\Delta}]\simeq \E[|W_n|_{\Delta_{s_1(n)}}]\simeq \sqrt{n}\, l_{s_1(n)}+\frac{n}{k_{s_1(n)+1}}.$$
\end{prop}

This proposition allows one to realize various functions as the speed functions of such groups. To do so, given a function $f\colon\N\to\N$ one must choose appropriate sequences $\set{k_s}$ and $\set{l_s}$. This is done in the following way:

\begin{prop}[{\cite[Corollary B.3]{brzh}}]\label{prop:approx-funcs}
Let $f\colon\N\to\N$ be a function such that $\frac{f(n)}{\sqrt{n}}$ and $\frac{n}{f(n)}$ are non-decreasing, and let $m_0>1$. Then one can find sequences $\set{k_s}$ and $\set{l_s}$ of positive integers such that $k_{s+1}\ge m_0k_s$, $l_{s+1}\ge m_0l_s$, and also such that
$$f(n)\simeq\bar{f}(n)$$
where
$$\bar{f}(x)=\sqrt{x}\, l_s+\frac{x}{k_{s+1}}\;\textrm{if }(k_sl_s)^2\leq x<(k_{s+1}l_{s+1})^2.$$
\end{prop}

This concludes the proof of \Tref{thm:brzh-main-thm}.

\section{Sketch of the proof}

As mentioned in the introduction, in order to show that
$$0<\limsup_{n\to\infty}\frac{|W_n|_{\Delta}}{a_n}<\infty$$ almost surely, we need to prove that there are constants $C,c>0$ such that that $Ca_n$ is an outer radius but $ca_n$ is not an outer radius. For the analogous claim
$$0<\liminf_{n\to\infty}\frac{|W_n|_{\Delta}}{b_n}<\infty$$ we need to prove that there are constants $C',c'>0$ such that $c'b_n$ is an inner radius for $|W_n|_{\Delta}$, but $C'b_n$ is not an inner radius.

Using a similar strategy as in \cite{brzh}, we first study the distance of the random walk on a single layer $\Delta_s$, and then deduce a bound for the distance of the random walk on $\Delta$. For a single layer, Brieussel and Zheng provide an upper bound for the distance of the random walk in terms of $k$-excursions \eqref{eq:brzh-eq1}, and a lower bound on the speed of the random walk
\begin{equation}\label{eq:brzh-eq2}
    \E\left[|W_n|_{\Delta_s}\right]\ge c\,\E\left[\sum_{x\in\Z}\min\set{T(k_s,x,n),l_s}\right].
\end{equation}

To estimate the fluctuations of the distance of the random walk, we use the same upper bound \eqref{eq:brzh-eq1}, and prove that \eqref{eq:brzh-eq2} holds also for the distance and not only for the expectation, i.e.\ for large enough~$n$
$$|W_n|_{\Delta_s}\ge c\sum_{x\in\Z}\min\set{T(k_s,x,n),l_s}.$$
As such, we study the tail behavior of $T(k,n)=k\sum_{x\in\Z}T(k,kx,n)$ and of $\sum_{x\in\Z}\min\set{T(k,x,n),l}$ for all $k,l$. This gives a bound on the distance in a single layer in terms of $k_s,l_s$, which we then add up by~\eqref{eq:brzh-eq0} to bound the distance of the random walk on $\Delta$.

Because of the Law of Iterated Logarithm on $\Z$, we know that there are times in which the range is small or large compared to its expected value. We show that in these times, the distance reaches its $\liminf$ and $\limsup$ values, respectively. As in the case of the speed of the random walk, we prove that there is one layer $\Delta_s$ that is equivalent to the distance in $\Delta$; however, since we consider the times in which the range is extremal, we need to choose different copies rather than $s_1(n)$.

The article is constructed as follows. In \Sref{sec:k-rw} we associate another random walk to $S_t$, so that the numbers of $k$-excursions of $S_t$ at all points correspond to the local times of the new random walk. To do so we use the results in \Aref{app:random-walk-facts}, which study the induced random walk of $S_t$ on $k\Z$. In \Sref{sec:total-ex} we study the tail behavior of $T(k,n)$, showing that under appropriate assumptions it is highly concentrated around its mean $\frac{n}{2k}$. In \Srefs{sec:max-ex}{sec:chop-ex} we study the tail behavior of $\sum_{x\in\Z}\min\set{T(k,x,n),l}$, for the two critical layers. These results rely on the tail behavior of the maximal local time of a random walk, which we study in \Aref{app:max-local-time-brow}.

We then return to the random walk on $\Delta$. In \Sref{sec:dist-bounds}, we prove the lower bound for the distance of the random walk in terms of the $k$-excursions. \Sref{sec:bounds-sing-layer} is devoted to deduce the bounds for the distance in a single layer, and in \Sref{sec:bounds-total} we use the bounds for a single layer to bound the distance of the random walk on $\Delta$. Finally, in \Sref{sec:main-thm-functions} we express our results in terms of the speed function, and conclude the proof of \Tref{thm:main-thm}.

\begin{rem}
In \Tref{thm:main-thm} we demand that the random walks are slightly super-diffusive. These conditions stem from the fact that the more diffusive the random walk is, the harder it is to get strong concentration inequalities for the number of $k$-excursions.

For the $\limsup$ statement (\Tref{thm:main-thm} \eqref{eq:limsup main}), this manifests in not having a precise enough bound for the contribution of the layers near the $\limsup$ critical layer. It may be that by having a better understanding of the joint behaviour of the distance in those layers one could improve the bounds.

For the $\liminf$ statement (\Tref{thm:main-thm} \eqref{eq:liminf main}), the difficulty arises from the sequence of parameters $k_s,l_s$ used in the approximation of the speed function. When $f$ is almost diffusive these sequences may grow extremely fast. One can still write an explicit formula for the $\liminf$ in terms of $k_s,l_s$ (see \Pref{prop:liminf-s-expr}), however this may not coincide with the expression in the theorem.
\end{rem}

\section{The $k$-induced random walk}\label{sec:k-rw}

Let $S_t$ be a simple random walk on $\Z$. Recall that $T(k,x,n)$ denotes the number of $k$-excursions starting from $x$ that are completed until time $n$. To study the distribution of $T(k,x,n)$ for each $k,x,n$, we associate a new simple random walk on~$\Z$ to $S_t$, so that $k$-excursions in $S_t$ can be approximated by the local times of the new random walk.

\begin{defn}\label{defn:k-assoc}
Let $S_t$ be a simple random walk on $\Z$, and let $k\ge 0$. The \textbf{k-induced random walk of $S_t$} is the simple random walk~$Y_n^{(k)}$ on $\Z$ defined in the following manner. We first set $n^{(k)}_0=0$, and for any $j\ge 1$ define inductively
$$n^{(k)}_j=\min\set{t>n^{(k)}_{j-1}\suchthat \left|S_t-S_{n^{(k)}_{j-1}}\right|=k}.$$
The random walk $Y_t^{(k)}$ is then given by $Y_t^{(k)}=S_{n^{(k)}_t}/k$. We also write $N_k(n)=\max\set{j\suchthat n^{(k)}_j\leq n}$ for the number of steps $Y_n^{(k)}$ made until time $n$.
\end{defn}

It is easy to see that $Y_n^{(k)}$ is indeed a simple random walk on $\Z$. In addition, as the expected time to reach $\pm k$ starting from $0$ is $k^2$, we have $\E[N_k(n)]=\frac{n}{k^2}$.

\begin{prop}\label{prop:nsteps}
There are universal constants $d_1,c,C>0$ such that for all $n$ and for all $k\leq\frac{d_1\sqrt{n}}{\sqrt{\log\log n}}$,
$$\Pr\left(N_k(n)\notin\left[\frac{cn}{k^2},\frac{Cn}{k^2}\right]\right)\leq\frac{2}{\log^3 n}.$$
\end{prop}
\begin{proof}
By choosing appropriate $c,c'>0$, \Aref{app:random-walk-facts} shows that
$$\Pr\left(N_k(n)\notin\left[\frac{cn}{k^2},\frac{Cn}{k^2}\right]\right)\leq 2\exp\left(-\frac{n}{20k^2}\right).$$
Note that if $k\leq\frac{\sqrt{n}}{8\sqrt{\log\log n}}$, we have $\frac{n}{k^2}\ge 64\log\log n$, so
$$\exp\left(-\frac{n}{20k^2}\right)\leq\exp\left(-3\log\log n\right)=\frac{1}{\log^3 n}$$
completing the proof.
\end{proof}

We turn to give bounds on the $k$-excursions of $S_t$ by means of the induced random walks. Denote the local time of $Y_n^{(k)}$ at $x$ by
$$L^{(k)}(x,n)=\left|\set{0\leq t\leq n\suchthat Y_t^{(k)}=x}\right|,$$
and let
$$\ell^{(k)}(x,n)=\left|\set{0\leq t<n\suchthat Y_t^{(k)}=x,Y_{t+1}^{(k)}=x-1}\right|.$$

\begin{prop}\label{prop:ex-bounds}
For any $x\in\Z$,
$$\ell^{(2k)}\left(\ceil{\frac{x}{2k}},n\right)-1\leq T(k,x,n)\leq L^{(k/2)}\left(\floor{\frac{x}{k/2}},n\right).$$
\end{prop}
We remark that when $\frac{k}{2}$ is not an integer, it can be replaced with $\floor{\frac{k}{2}}$ and the proposition still holds. However, we treat $\frac{k}{2}$ as an integer for convenience.
\begin{proof}
For the upper bound, we write $x=\frac{k}{2}j+r$ with $0\leq r<\frac{k}{2}$. Any $k$-excursion starting at $x$ must include at least one $\frac{k}{2}$-excursion starting at $\frac{k}{2}j$, so $T(k,x,n)\leq T\left(\frac{k}{2},\frac{k}{2}j,n\right)\leq L^{(k/2)}(j,n)$.

For the lower bound, write $x=2kj'+r'$ with $0\leq r'<2k$. A similar argument shows that $T(2k,2k(j'+1),n)\leq T(k,x,n)$. We note that $\ell^{(2k)}(2k(j'+1),n)$ counts the number of $2k$-excursions starting at $2k(j'+1)$ which are completed before time $n$, with the possibility that we don't finish the last $2k$-excursion before time $n$. This proves the lower bound.
\end{proof}

\section{The total number of $k$-excursions}\label{sec:total-ex}

We turn to give upper and lower bounds for
$$T(k,n)=\sum_{x\in\Z}kT(k,kx,n).$$
Note that this also provides bounds on the total number of $k$-excursions, by the following proposition:
\begin{prop}\label{prop:Tkn-and-sum}
For any $n,k$,
$$\frac{1}{2}T(2k,n)\leq\sum_{x\in\Z}T(k,x,n)\leq T\left(\frac{k}{2},n\right).$$
\end{prop}
\begin{proof}
For the lower bound, note that any $2k$-excursion starting at $2kx$ must contain at least one $k$-excursion from each of $2kx-k,\dots,2kx$. Therefore
$$T(2k,2kx,n)\leq\frac{1}{k}\sum_{j=0}^{k-1}T(k,2kx-j,n)$$
and we have
$$T(2k,n)=2k\sum_{x\in\Z}T(2k,2kx,n)\leq 2\sum_{x\in\Z}T(k,x,n).$$
For the upper bound, notice that any $k$-excursion from $\frac{k}{2}x+j$ for some $0\leq j<\frac{k}{2}$ contains at least one $\frac{k}{2}$-excursion from $\frac{k}{2}x$. Therefore
\begin{align*}
   \sum_{x\in\Z}T(k,x,n)&=\sum_{x\in\Z}\sum_{j=0}^{k/2-1}T\left(k,\frac{k}{2}x+j,n\right)\leq\\
   &\leq\frac{k}{2}\sum_{x\in\Z}T\left(\frac{k}{2},\frac{k}{2}x,n\right)=T\left(\frac{k}{2},n\right).
\end{align*}
\end{proof}

Before studying the asymptotic behavior of $T(k,n)$, we prove an (almost) monotonicity result for $T(k,n)$:
\begin{lem}\label{lem:ex-monotone}
Suppose $2k\leq k'$. Then $T(k',n)\leq 3T(k,n)$.
\end{lem}
\begin{proof}
For any $x\in\Z$, each $k'$-excursion from $k'x$ contains at least $\floor{\frac{k'}{k}}-1$ many $k$-excursions from points $k\Z$. Therefore
\begin{align*}
T'(k',n)&=k'\sum_{x\in\Z}T(k',k'x,n)\leq k'\frac{1}{\floor{\frac{k'}{k}}-1}\sum_{x\in\Z}T(k,kx,n)=\\
&=\frac{\frac{k'}{k}}{\floor{\frac{k'}{k}}-1}T(k,n)\leq 3T(k,n).\qedhere
\end{align*}
\end{proof}

For specific values of $n$ and $k$ of order at most $\frac{\sqrt{n}}{\sqrt{\log\log n}}$, we use the $k$-induced random walks defined in \Sref{sec:k-rw} to get a concentration result:
\begin{lem}\label{lem:bound-Tkn-single}
There are universal constants $d_2,c,C>0$ such that the following holds: For large enough $n$, if $k\leq\frac{d_2\sqrt{n}}{\sqrt{\log\log n}}$,
$$\Pr\left(T(k,n)\notin\left[\frac{cn}{k},\frac{Cn}{k}\right]\right)\leq\frac{6}{\log^3 n}.$$
\end{lem}
\begin{proof}
Using \Pref{prop:nsteps}, let $d_1,c,C>0$ be the constants such that
$$\Pr\left(N_k(n)\notin\left[\frac{cn}{k^2},\frac{Cn}{k^2}\right]\right)\leq\frac{2}{\log^3 n}$$
for all $k\leq\frac{d_1\sqrt{n}}{\sqrt{\log\log n}}$. Then for any $k\leq\frac{d_1\sqrt{n}}{\sqrt{\log\log n}}$, by \Pref{prop:ex-bounds},
\begin{align*}
  \Pr\left(T(k,n)>\frac{4Cn}{k}\right) & \leq\Pr\left(k\sum_{x\in\Z}L^{(k/2)}(x,n)>\frac{4Cn}{k}\right)=\\
  &=\Pr\left(N_{k/2}(n)>\frac{Cn}{(k/2)^2}\right)\leq\frac{2}{\log^3 n}.
\end{align*}

For the lower bound, we use again \Pref{prop:ex-bounds},
\begin{align*}
  &\Pr\left(T(k,n)<\frac{k}{8}N_{2k}(n)\suchthat N_{2k}(n)\right) \leq \\
  &\leq\Pr\left(k\sum_{x\in\Z}(\ell^{(2k)}(x,n)-1)<\frac{k}{8}N_{2k}(n)\suchthat N_{2k}(n)\right)\leq\\
  &\leq\Pr\left(\sum_{x\in\Z}(\ell^{(2k)}(x,n)-1)<\frac{1}{8}N_{2k}(n)\suchthat N_{2k}(n)\right).
\end{align*}

Let $Z=\left|\set{x:L^{(2k)}(x,n)>0}\right|$. This is the range of a simple random walk on $\Z$ after $N_{2k}(n)$ steps. We thus have
\begin{align*}
\Pr\left(Z\ge\frac{1}{8}N_{2k}(n)\suchthat N_{2k}(n)\right)&\leq\exp\left(-\frac{(N_{2k}(n)/8)^2}{2N_{2k}(n)}\right)=\\
&=\exp\left(-\frac{N_{2k}(n)}{128}\right),
\end{align*}
so
\begin{align*}
  &\Pr\left(T(k,n)<\frac{k}{8}N_{2k}(n)\suchthat N_{2k}(n)\right) \leq \\
  &\leq\Pr\left(\sum_{x\in Z}\ell^{(2k)}(x,n)<\frac{1}{8}N_{2k}(n)+Z\suchthat N_{2k}(n)\right)\leq\\
  &\leq\Pr\left(\sum_{x\in Z}\ell^{(2k)}(x,n)<\frac{1}{4}N_{2k}(n)\suchthat N_{2k}(n)\right)+\exp\left(-\frac{N_{2k}(n)}{128}\right).
\end{align*}

Conditioning on $N_{2k}(n)$, $\sum_{x\in Z}\ell^{(2k)}(x,n)$ is a binomial random variable, with parameters $N_{2k}(n)$ and $\frac{1}{2}$. Therefore, by Chernoff's inequality,
$$\Pr\left(\sum_{x\in Z}\ell^{(2k)}(x,n)<\frac{1}{4}N_{2k}(n)\suchthat N_{2k}(n)\right)\leq\exp\left(-\frac{1}{16}N_{2k}(n)\right).$$

We therefore have
$$\Pr\left(T(k,n)<\frac{k}{8}N_{2k}(n)\suchthat N_{2k}(n)\right)\leq 2\exp\left(-\frac{1}{128}N_{2k}(n)\right),$$
so
\begin{align*}
  & \Pr\left(T(k,n)<\frac{cn}{16k}\right)\leq\\
  & \leq\Pr\left(N_{2k}(n)<\frac{cn}{2k^2}\right)+ \Pr\left(T(k,n)<\frac{cn}{16k}\suchthat N_{2k}(n)\ge\frac{cn}{2k^2}\right)\leq\\
  &\leq\frac{2}{\log^3 n}+ \Pr\left(T(k,n)<\frac{k}{8}N_{2k}(n)\suchthat N_{2k}(n)\ge\frac{cn}{2k^2}\right)\leq\\
  &\leq\frac{2}{\log^3 n}+2\exp\left(-\frac{cn}{256k^2}\right).
\end{align*}
Let $d_2>0$ be a constant such that $d_2^2\leq \sqrt{c/768}$. For any $k\leq\frac{d_2\sqrt{n}}{\sqrt{\log\log n}}$,
$$\frac{n}{k^2}\ge\frac{1}{d_2^2}\log\log n\ge\frac{768}{c}\log\log n,$$
and thus
$$\Pr\left(T(k,n)<\frac{cn}{16k}\right)\leq\frac{4}{\log^3 n}$$
concluding the proof.
\end{proof}

We are now ready to get a concentration result for all values of $n,k$ simultaneously:

\begin{prop}\label{prop:bound-Tkn-as}
There are universal constants $d_2,c_1,C_1>0$ such that the following holds almost surely: for all but finitely many $n$ and for all $k\leq\frac{d_2\sqrt{n}}{2\sqrt{\log\log n}}$,
$$\frac{c_1n}{k}\leq T(k,n)\leq\frac{C_1n}{k}.$$
\end{prop}
\begin{proof}
We choose an exponential sequence of times $t_m=2^m$, and for each $m$ an exponential sequence $k_{m,i}=2^i$ defined for all $i$ such that $k_{m,i}\leq\frac{d_2\sqrt{t_m}}{\sqrt{\log\log t_m}}$. Note that for each $m$ there are at most $K\log t_m$ such indices for some constant $K$.

By \Lref{lem:bound-Tkn-single}, there are constants $c,C>0$ such that for each $m,i$,
$$\Pr\left(T(k_{m,i},t_m)\notin\left[\frac{ct_m}{k_{m,i}},\frac{Ct_m}{k_{m,i}}\right]\right)\leq\frac{6}{\log^3 t_m}.$$
Taking a union bound over $i$, we have
$$\Pr\left(\exists i:T(k_{m,i},t_m)\notin\left[\frac{ct_m}{k_{m,i}},\frac{Ct_m}{k_{m,i}}\right]\right)\leq\frac{6K}{\log^2 t_m}=\frac{6K}{m^2\log^22}.$$

As the latter expression is summable over $m$, by the Borel-Canetlli lemma we have that for all but finitely many $m$, for all $i$
\begin{equation}\label{eq:ex-b1}
\frac{ct_m}{k_{m,i}}\leq T(k_{m,i},t_m)\leq\frac{Ct_m}{k_{m,i}}.
\end{equation}
We extend this result to general $n,k$ in two steps.

First, let $m$ be such that \eqref{eq:ex-b1} holds for all $i$, and let $k\leq\frac{d_2\sqrt{t_m}}{2\sqrt{\log\log t_m}}$. Take $i$ such that $k_{m,i}\leq k<k_{m,i+1}$. By \Lref{lem:ex-monotone},
$$T(k,t_m)\leq 3T(k_{m,i-1},t_m)\leq\frac{3Ct_m}{k_{m,i-1}}\leq\frac{6Ct_m}{k}$$
and
$$T(k,t_m)\ge\frac{1}{3}T(k_{m,i+1},t_m)\ge\frac{ct_m}{3k_{m,i+1}}\ge\frac{ct_m}{6k}.$$

Finally, take any $n$, and let $m$ be such that $t_m\leq n<t_{m+1}$. Suppose $n$ is large enough so that \eqref{eq:ex-b1} holds. Then
$$T(k,n)\leq T(k,t_{m+1})\leq\frac{6Ct_{m+1}}{k}\leq\frac{12Cn}{k}$$
and
$$T(k,n)\ge T(k,t_m)\ge\frac{ct_m}{6k}\ge\frac{cn}{12k}$$
as required.
\end{proof}

\section{The truncated sum of $k$-excursions, lower layer}\label{sec:max-ex}

We turn to study the sum $\sum_{x\in\Z}\min\set{T(k,x,n),l}$ for $k,l>0$. Our two cases of interest are when $kl\sim\sqrt{n\log\log n}$ (which will turn out to be the $\liminf$ layer) and when $kl\sim\frac{\sqrt{n}}{\sqrt{\log\log n}}$ (which will be the $\limsup$ layer). We begin with the $\liminf$ layer; for this we study the maximal number of $k$-excursions from a given point, i.e.\ $\max_{x\in\Z}T(k,x,n)$.

\subsection{The maximal number of $k$-excursions}

Similarly to the way we estimated $T(k,n)$, we get tail bounds on the maximal number of $k$-excursions in two steps. We first use the induced random walks to get a tail bound for given values of $n$ and $k$, and then use exponential scales to combine these bounds into a bound for the maximal number of $k$-excursions.

\begin{lem}\label{lem:max-ex-pr}
There is a universal constant $c>0$ such that for all large enough $n$ and for all $k\leq\frac{d_1\sqrt{n}}{\sqrt{\log\log n}}$,
$$\Pr\left(\max_{x\in\Z}T(k,x,n)\ge\frac{c\sqrt{n\log\log n}}{k}\right)\leq\frac{3}{\log^3 n}.$$
\end{lem}
\begin{proof}
We use the induced random walks from \Sref{sec:k-rw}. By \Pref{prop:ex-bounds},
\begin{align*}
  & \Pr\left(\max_{x\in\Z}T(k,x,n)\ge\frac{c\sqrt{n\log\log n}}{k}\right)\leq\\
  &\leq\Pr\left(\max_{x\in\Z}L^{(k/2)}(x,n)\ge\frac{c\sqrt{n\log\log n}}{k}\right).
\end{align*}
From \Aref{app:max-local-time-brow},
$$\Pr\left(\max_{x\in\Z}L(x,n)\ge xn^{1/2}\right)\leq Cx^2\exp\left(-C'x^2\right)$$
for large enough $n,x$. Let $c'>0$ be a constant such that for $k\leq\frac{d_1\sqrt{n}}{\sqrt{\log\log n}}$
$$\Pr\left(N_{k/2}(n)\ge\frac{c'n}{k^2}\right)\leq\frac{2}{\log^3 n}$$
Using \Pref{prop:nsteps},
\begin{align*}
  &\Pr\left(\max_{x\in\Z}L^{(k/2)}(x,n)\ge\frac{c\sqrt{n\log\log n}}{k}\right)\leq\\
  &\leq\Pr\left(N_{k/2}(n)\ge\frac{c'n}{k^2}\right)+\Pr\left(\max_{x\in\Z}L\left(x,\frac{c'n}{k^2}\right)\ge\frac{c\sqrt{n\log\log n}}{k}\right)\leq\\
  &\leq\frac{2}{\log^3 n}+\frac{c^2}{c'}\log\log n\exp\left(-\frac{c^2}{c'}\log\log n\right).
\end{align*}
Finally, choosing $c=2\sqrt{c'}$, we have
$$\frac{c^2}{c'}\log\log n\exp\left(-\frac{c^2}{c'}\log\log n\right)=\frac{2\log\log n}{\log^4 n}\leq\frac{1}{\log^3 n}$$
for large enough $n$, as required.
\end{proof}

\begin{prop}\label{prop:max-ex-fin}
There are universal constants $d_1,C>0$ such that the following holds almost surely: For all but finitely many $n$ and for all $k\leq\frac{d_1\sqrt{n}}{2\sqrt{\log\log n}}$,
$$\max_{x\in\Z}T(k,x,n)\leq\frac{C\sqrt{n\log\log n}}{k}.$$
\end{prop}
\begin{proof}
We choose an exponential sequence of times $t_m=2^m$, and for each $m$ an exponential sequence $k_{m,i}=2^i$ defined for all $i$ such that $k_{m,i}\leq\frac{\sqrt{t_m}}{2\sqrt{\log\log t_m}}$. Note that for each $m$ there are at most $\log t_m$ such indices.

By \Lref{lem:max-ex-pr}, there is a constant $c>0$ such that for all $m,i$ we have
$$\Pr\left(\max_{x\in\Z}T(k_{m,i},x,t_m)\ge\frac{c\sqrt{t_m\log\log t_m}}{k_{m,i}}\right)\leq\frac{3}{\log^3 t_m}.$$
Taking a union bound,
$$\Pr\left(\exists i:\max_{x\in\Z}T(k_{m,i},x,t_m)\ge\frac{c\sqrt{t_m\log\log t_m}}{k_{m,i}}\right)\leq\frac{3}{\log^2 t_m}=\frac{3}{m^2\log^22}.$$
The latter expression is summable over $m$, so by Borel-Cantelli lemma we have that for all but finitely many $m$, for all $i$ we have
\begin{equation}\label{eq:max-ex1}
  \max_{x\in\Z}T(k_{m,i},x,t_m)\leq\frac{c\sqrt{t_m\log\log t_m}}{k_{m,i}}.
\end{equation}

Let $n,k$ be such that $k\leq\frac{d_1\sqrt{n}}{2\sqrt{\log\log n}}$. Take $m$ such that $t_m\leq n<t_{m+1}$, and suppose $n$ is large enough so that \eqref{eq:max-ex1} holds. Take $i$ such that $k_{m,i}\leq k<k_{m,i+1}$. Then
\begin{align*}
\max_{x\in\Z}T(k,x,n)&\leq\max_{x\in\Z}T(k_{m,i},x,t_{m+1})\leq\frac{c\sqrt{t_{m+1}\log\log t_{m+1}}}{k_{m,i}}\leq\\
&\leq\frac{4c\sqrt{n\log\log n}}{k}
\end{align*}
as required.
\end{proof}

\subsection{The truncated sum}

We are now ready to get an asymptotic bound for the truncated sum of $k$-excursions.
\begin{prop}\label{prop:chop-sum-as-bound}
There are universal constants $d_2,c_2>0$ such that the following holds almost surely: for all but finitely many $n$, for all $k\leq\frac{d_2\sqrt{n}}{2\sqrt{\log\log n}}$ and for all $l$,
$$\sum_{x\in\Z}\min\set{T(k,x,n),l}\ge c_2\min\set{\frac{n}{k},\frac{\sqrt{n}}{\sqrt{\log\log n}}l}.$$
\end{prop}
\begin{proof}
We first assume that $l\ge\frac{\sqrt{n\log\log n}}{k}$. By \Pref{prop:max-ex-fin}, there is a constant $C>0$ such that for large enough $n$ we have
$$\max_{x\in\Z}T(k,x,n)\leq\frac{C\sqrt{n\log\log n}}{k}\leq Cl,$$
so by \Pref{prop:bound-Tkn-as}
$$\sum_{x\in\Z}\min\set{T(k,x,n),l}\ge\frac{1}{C}\sum_{x\in\Z}T(k,x,n)\ge\frac{1}{2C}T(2k,n)\ge\frac{c'n}{k}$$
for large enough $n$ and a universal constant $c'$.

Now suppose $l<\frac{\sqrt{n\log\log n}}{k}$. Let
$$A_1=\set{x\in\Z\suchthat T(k,x,n)<l},\;\;\;A_2=\set{x\in\Z\suchthat T(k,x,n)\ge l}.$$
Note that by \Pref{prop:Tkn-and-sum} and \Pref{prop:bound-Tkn-as}, for large enough $n$ we have
$$\frac{c_1n}{4k}\leq \frac{1}{2}T(2k,n)\leq\sum_{x\in\Z}T(k,x,n)=\sum_{x\in A_1}T(k,x,n)+\sum_{x\in A_2}T(k,x,n).$$
We split into two cases:
\begin{enumerate}
  \item If $\sum_{x\in A_1}T(k,x,n)\ge\frac{c_1n}{8k}$, then clearly
      $$\sum_{x\in\Z}\min\set{T(k,x,n),l}\ge\sum_{x\in A_1}T(k,x,n)\ge\frac{c_1n}{8k}\ge\frac{c_1\sqrt{n}}{8\sqrt{\log\log n}}l.$$
  \item Otherwise, we must have $\sum_{x\in A_2}T(k,x,n)\ge\frac{c_1n}{8k}$. By \Pref{prop:max-ex-fin},
      $$\frac{c_1n}{8k}\leq\sum_{x\in A_2}T(k,x,n)\leq\frac{c\sqrt{n\log\log n}}{k}\cdot\left|A_2\right|,$$
      so $\left|A_2\right|\ge\frac{c_1\sqrt{n}}{8c\sqrt{\log\log n}}$, showing
      $$\sum_{x\in\Z}\min\set{T(k,x,n),l}\ge l\cdot\left|A_2\right|\ge\frac{c_1\sqrt{n}}{8c\sqrt{\log\log n}}l.$$
\end{enumerate}
The proposition is thus proved.
\end{proof}

\section{The truncated sum of $k$-excursions, upper layer}\label{sec:chop-ex}

For the remaining case, we switch strategy. We first need the exact distribution of $T(k,0,n)$, which can be found e.g.\ in \cite{hucs}. The main idea is the following: We use the reflection principle successively, reflecting the random walk whenever we first visit any $jk$ for any $j<0$. In this manner, the number of $k$-excursions starting from $0$ has the same distribution (up to a factor of $2k$) as the minimum (or maximum) of the random walk. This proves the following:
\begin{lem}[{\cite{hucs}}]\label{lem:ex-dist}
For any $a\in\N$,
$$\Pr\left(T(k,0,n)\ge a\right)=\Pr\left(\min_{0\leq t\leq n}S_t\leq -2ka\right)=\Pr\left(\max_{0\leq t\leq n}S_t\ge 2ka\right).$$
\end{lem}

\begin{lem}\label{lem:ex-point-small-bound}
There are universal constants $c,c'>0$ such that for large enough $n$ and for all $k\leq\frac{c\sqrt{n}}{\sqrt{\log\log n}}$,
$$\Pr\left(T(k,0,n)\leq\frac{c\sqrt{n}}{k\sqrt{\log\log n}}\right)\leq\frac{c'}{\log^3 n}.$$
\end{lem}
\begin{proof}
By the reflection principle (see \Lref{lem:ex-dist}),
\begin{align*}
  \Pr\left(T(k,0,n)\leq\frac{c\sqrt{n}}{k\sqrt{\log\log n}}\right) & \leq \Pr\left(T(k,0,n)\leq\ceil{\frac{c\sqrt{n}}{k\sqrt{\log\log n}}}\right)=\\
  &=\Pr\left(\max_{0\leq t\leq n}S_t\leq 2k\ceil{\frac{c\sqrt{n}}{k\sqrt{\log\log n}}}\right)\leq\\
  &\leq\Pr\left(\max_{0\leq t\leq n}S_t\leq \frac{2c\sqrt{n}}{\sqrt{\log\log n}}+2k\right)\leq\\
  &\leq\Pr\left(\max_{0\leq t\leq n}S_t\leq \frac{4c\sqrt{n}}{\sqrt{\log\log n}}\right)
\end{align*}
where the last inequality follows from $k\leq \frac{c\sqrt{n}}{\sqrt{\log\log n}}$. To bound the latter probability, we use \cite[Lemma 2]{japr}. Following their proof, one can choose $c=\frac{\pi}{4}$ and get
$$\Pr\left(\max_{0\leq t\leq n}S_t\leq\frac{\pi\sqrt{n}}{4\sqrt{\log\log n}}\right)\leq\frac{c'}{\log^3 n}$$
for large enough $n$, concluding our proof.
\end{proof}

\begin{prop}\label{prop:chop-sum-range}
There exists a universal constant $c>0$ such that the following holds almost surely: For all but finitely many values of $n$ and for all $k\leq\frac{c\sqrt{n}}{2\sqrt{\log\log n}}$,
$$\sum_{x\in\Z}\min\set{T(k,x,n),\frac{c\sqrt{n}}{k\sqrt{\log\log n}}}\ge\frac{c\sqrt{n}}{k\sqrt{\log\log n}}\cdot\left|\range(S_{n/4})\right|.$$
\end{prop}
\begin{proof}
For convenience, write $l=\frac{c\sqrt{n}}{k\sqrt{\log\log n}}$. We first note that
$$\sum_{x\in\Z}\min\set{T(k,x,n),l}\ge\sum_{x\in\range(S_{n/2})}\min\set{T(k,x,n),l}.$$
Writing
$$Z_{n,k}=\sum_{x\in\range(S_{n/2})}\max\set{l-T(k,x,n),0},$$
we have
\begin{equation}\label{eq:chop-sup}
\sum_{x\in\range(S_{n/2})}\min\set{T(k,x,n),l}=l\cdot\left|\range(S_{n/2})\right|-Z_{n,k},
\end{equation}
so we turn to upper bound $Z_{n,k}$.

Note that for any $x\in\range(S_{n/2})$, by \Lref{lem:ex-point-small-bound}
$$\Pr\left(T(k,x,n)\leq l\right)\leq\Pr\left(T(k,0,n/2)\leq l\right)\leq\frac{c'}{\log^3 n}$$
for large enough $n$. Therefore, for any $x\in\range(S_{n/2})$,
$$\E[\max\set{l-T(k,x,n),0}]\leq l\cdot\Pr\left(T(k,x,n)\leq l\right)\leq\frac{c'}{\log^3 n}l,$$
and thus
$$\E[Z_{n,k}]\leq\frac{c'}{\log^3 n}l\cdot\E\left[\range(S_{n/2})\right]\leq\frac{c''\sqrt{n}}{\log^3 n}l.$$
By Markov's inequality,
$$\Pr\left(Z_{n,k}\ge\frac{\sqrt{n}}{4\sqrt{\log\log n}}l\right)\leq\frac{\frac{c''\sqrt{n}}{\log^3 n}l}{\frac{\sqrt{n}}{4\sqrt{\log\log n}}l}\leq\frac{4c''}{\log^{2.5}n}$$
for large enough $n$.

Let $t_m=2^m$ and let $k_{m,i}=\frac{c\sqrt{t_m}}{2^i\sqrt{\log\log t_m}}$ (so $l_{m,i}=\frac{c\sqrt{t_m}}{k_{m,i}\sqrt{\log\log t_m}}=2^i$) for all $i$ such that $k_{m,i}\ge 1$. As there are at most $\log t_m$ such indices, we have
$$\Pr\left(\exists i:Z_{t_m,k_{m,i}}\ge\frac{\sqrt{t_m}}{4\sqrt{\log\log t_m}}l_{m,i}\right)\leq\frac{4c''}{\log^{1.5}t_m}=\frac{4c''}{m^{1.5}\log^{1.5}2}.$$
The latter expression is summable over $m$, so by Borel-Cantelli Lemma we have that for all but finitely many $m$ and for all $i$ as above,
$$Z_{t_m,k_{m,i}}\leq\frac{\sqrt{t_m}}{4\sqrt{\log\log t_m}}l_{m,i}.$$
From the $\liminf$ Law of Iterated Logarithm (see \cite{chu, japr}), we know that for large enough $m$, $\left|\range(S_{t_m/2})\right|>\frac{\sqrt{t_m}}{2\sqrt{\log\log t_m}}$, and thus \eqref{eq:chop-sup} yields
\begin{equation}\label{eq:chop-sup2}
\sum_{x\in\Z}\min\set{T(k_{m,i},x,t_m),l_{m,i}}\ge\frac{l_{m,i}}{4}\cdot\left|\range(S_{t_m/2})\right|
\end{equation}
for all but finitely many $m$ and for all $i$ as above.

Now, let $n$, let $k\leq\frac{c\sqrt{n}}{2\sqrt{\log\log n}}$, and write $l=\frac{c\sqrt{n}}{k\sqrt{\log\log n}}$. Let $m$ such that $t_m\leq n<t_{m+1}$, and let $i$ such that $k_{m,i}\leq k<k_{m,i+1}$. Assume $n$ is large enough so that \eqref{eq:chop-sup2} holds. Then
\begin{align*}
  \sum_{x\in\Z}\min\set{T(k,x,n),l} & \ge\sum_{x\in\Z}\min\set{T(k_{m,i+1},x,t_m),l_{m,i+1}}\ge\\
  &\ge\frac{l_{m,i+1}}{4}\cdot\left|\range(S_{t_m/2})\right|\ge\frac{l}{8}\cdot\left|\range(S_{n/4})\right|=\\
  &=\frac{c\sqrt{n}}{8k\sqrt{\log\log n}}\cdot\left|\range(S_{n/4})\right|.
\end{align*}
\end{proof}

\begin{cor}\label{cor:chop-sum}
There are universal constant $d_3,c_3>0$ such that the following holds almost surely: for infinitely many $n$ and for all $k,l$ with $k\leq\frac{d_3\sqrt{n}}{2\sqrt{\log\log n}}$,
$$\sum_{x\in\Z}\min\set{T(k,x,n),l}\ge c_3\min\set{\frac{n}{k},\sqrt{n\log\log n}\,l}.$$
\end{cor}
\begin{proof}
By the usual Law of Iterated Logarithm, there are infinitely many $n$ such that $\left|\range(S_{n/4})\right|\ge \frac{1}{4}\sqrt{n\log\log n}$. Using \Pref{prop:chop-sum-range}, we have two cases:
\begin{enumerate}
  \item If $l\ge\frac{\sqrt{n}}{k\sqrt{\log\log n}}$, then
  \begin{align*}
  \sum_{x\in\Z}\min\set{T(k,x,n),l}&\ge\sum_{x\in\Z}\min\set{T(k,x,n),\frac{\sqrt{n}}{k\sqrt{\log\log n}}}\ge\\
  &\ge\frac{c\sqrt{n}}{4k\sqrt{\log\log n}}\sqrt{n\log\log n}=\frac{cn}{4k}.
  \end{align*}
  \item If $l<\frac{\sqrt{n}}{k\sqrt{\log\log n}}$, let $k'=\frac{\sqrt{n}}{l\sqrt{\log\log n}}>k$. Then
  \begin{align*}
  \sum_{x\in\Z}\min\set{T(k,x,n),l}&\ge\sum_{x\in\Z}\min\set{T(k',x,n),l}\ge\\
  &\ge\frac{c\sqrt{n}}{4k'\sqrt{\log\log n}}\sqrt{n\log\log n}=\\
  &=\frac{c}{4}\sqrt{n\log\log n}\, l.
  \end{align*}
\end{enumerate}
This concludes the proof.
\end{proof}

\section{Bounds on the distance in $\Delta$}\label{sec:dist-bounds}

We return to the construction of \cite{brzh}. Recall that our group $\Delta$ is a diagonal product of lamplighter groups $\Delta_s=\Gamma_s\wr\Z$, where each $\Gamma_s$ is generated by a set of the form $A(s)\cup B(s)$, and the diagonal product is taken with respect to the generating set $\alpha_i(s)=(a_i(s)\delta_0,0)$, $\beta_i(s)=(b_i(s)\delta_{k_s},0)$ and $\tau(s)=(e,+1)$.

As explained in \Ssref{subs:speed-single-layer}, for any given layer $s$, the number of $\Gamma_s$-steps the random walk makes at $x$ in $\Delta_s$ depends only on whether we reached $x$ and on the number of $k_s$-excursions from $x$. In this section we provide bounds for the distance of a simple random walk on $\Delta$ and on $\Delta_s$ in terms of the $k_s$-excursions.

For the upper bound on the distance, we use the bound found in \cite[Proposition 2.14 and proof of Lemma 3.4]{brzh}:
\begin{prop}\label{prop:up-dist-bound}
For all $s\ge 0$,
$$|W_n|_{\Delta_s}\leq 11\min\set{\sum_{j\in\Z}k_sT\left(\frac{k_s}{2},j\frac{k_s}{2},n\right)+\left|\range(S_n)\right|,\left|\range(S_n)\right|\, l_s}.$$
In addition, writing $s_0(n)=\max\set{s\ge 0\suchthat k_s\leq\left|\range(S_n)\right|}$, we have
$$|W_n|_{\Delta}\leq 500\sum_{s\leq s_0(n)}|W_n|_{\Delta_s}.$$
\end{prop}

We now turn to the lower bound. For all $s\ge 0$ we have
$$|W_n|_\Delta\ge|W_n|_{\Delta_s}\ge\sum_{x\in\Z}\left|f_n(x)\right|_{\Gamma_s}$$

To get a lower bound for the RHS, we use the following lemma:

\begin{lem}\label{lem:prob-in-expand}
Let $\set{\Gamma_s}$ be a sequence of finite groups, let $l_s=\diam(\Gamma_s)$, and suppose that $\set{\Gamma_s}$ satisfies $(\sigma,c_0l_s)$-linear speed assumption for some constants $0<\sigma,c_0<1$ with $\sigma c_0l_s\ge 4$. Then for all $t\ge 0$,
$$\Pr\left(\left|X_t^{(s)}\right|_{\Gamma_s}\ge\frac{\sigma}{8}\min\set{t,c_0l_s}\right)\ge\frac{\sigma c_0}{8}.$$
\end{lem}
\begin{proof}
Assume first that $t\leq c_0l_s$. Let $p=\Pr\left(\left|X_t^{(s)}\right|_{\Gamma_s}\ge\frac{\sigma}{2} t\right)$. Then
$$\sigma t\leq\E\left|X_t^{(s)}\right|_{\Gamma_s}\leq (1-p)\frac{\sigma}{2} t+pt\leq\frac{\sigma}{2}t+pt,$$
so $p\ge\frac{\sigma}{2}$, and the assertion follows.

Now, let $t\ge c_0l_s$. By \cite[Lemma 4.1]{leeper}, we have
$$\E\left|X_t^{(s)}\right|_{\Gamma_s}\ge\frac{1}{2}\E\left|X_{c_0l_s}^{(s)}\right|_{\Gamma_s}-1 \ge\frac{\sigma}{2}c_0l_s-1.$$
If $\sigma c_0l_s\ge 4$, we therefore have
$$\E\left|X_t^{(s)}\right|\ge\frac{\sigma c_0l_s}{4}.$$
Let $p=\Pr\left(\left|X_t^{(s)}\right|\ge\frac{\sigma c_0l_s}{8}\right)$. Then
$$\frac{\sigma c_0l_s}{4}\leq\E\left|X_t^{(s)}\right|\leq(1-p)\frac{\sigma c_0l_s}{8}+pl_s\leq\frac{\sigma c_0l_s}{8}+pl_s.$$
We therefore have $p\ge\frac{\sigma c_0}{8}$, as required.
\end{proof}

We are now ready to prove our lower bound on the metric in a single layer:
\begin{prop}\label{prop:metric-lower-bound}
Let $\set{\Gamma_s}$ be a sequence of finite groups, let $l_s=\diam(\Gamma_s)$, and suppose that $\set{\Gamma_s}$ satisfies $(\sigma,c_0l_s)$-linear speed assumption for some constants $\sigma,c_0>0$. Also suppose there is $m_0>1$ such that
$$k_{s+1}>2k_s\;\;\;\textrm{and}\;\;\;l_{s+1}\ge m_0l_s$$
for all $s$. Then, almost surely, for large enough~$n$, for all $s$ such that $k_s\leq\frac{d_2\sqrt{n}}{\sqrt{\log\log n}}$ we have
$$\left|W_n\right|_{\Delta_s}\ge\frac{\sigma c_0}{16}\sum_{x\in\Z}\min\set{T(k_s,x,n),c_0l_s}.$$
\end{prop}
\begin{proof}
We use a similar approach to Revelle \cite{rev}. It is trivial that for all $n,s$,
$$\left|W_n\right|_{\Delta_s}\ge\frac{1}{2}\sum_{x\in\Z}\left|f_n(x)\right|_{\Gamma_s}.$$
Let $I_{x,n}$ denote the indicator of whether the random walk on $\Z$ visited~$x$ before time $n$. Conditioning on $T(k_s,x,n)$ and $I_{x,n}$ for all $x$, the random variables $\left|f_n(x)\right|_{\Gamma_s}$ are independent.

Let $\theta_x=\frac{\sigma}{8}\min\set{T(k_s,x,n),c_0l_s}$, and write $Z_x=\left|f_n(x)\right|_{\Gamma_s}\theta_x^{-1}$. By \Lref{lem:prob-in-expand}, $\Pr\left(Z_x\ge 1\right)\ge\frac{\sigma c_0}{8}$, so we may apply \cite[Lemma 2]{rev} to get
\begin{align}\label{eq:metric-lower-bound-eq1}
&\Pr\left(\sum_{x\in\Z}\left|f_n(x)\right|_{\Gamma_{s}}<\frac{\sigma c_0}{16}\sum_{x\in\Z}\theta_x\suchthat T(k_{s},x,n),I_{x,n}\right)=\\ &=\Pr\left(\sum_{x\in\Z}\frac{\left|f_n(x)\right|_{\Gamma_{s}}}{\theta_x}\frac{\theta_x}{\max_x\theta_x}< \frac{\sigma c_0}{16}\sum_{x\in\Z}\frac{\theta_x}{\max_x\theta_x}\suchthat T(k_{s},x,n),I_{x,n}\right)\leq\notag\\
&\leq\exp\left(-\frac{\sigma c_0}{96\max_x\theta_x}\sum_{x\in\Z}\theta_x\suchthat T(k_s,x,n),I_{x,n}\right).\notag
\end{align}

Let $d_2,c,c_2>0$ be constants such that \Pref{prop:max-ex-fin} and \Pref{prop:chop-sum-as-bound} hold. For each $m$, let $A_m$ denote the event that for all $n\ge m$ and for all $s$ with $k_s\leq\frac{d_2\sqrt{n}}{2\sqrt{\log\log n}}$,
$$\max_{x\in\Z}\theta_x\leq\min\set{\frac{c\sqrt{n\log\log n}}{k_s},l_s}$$
and
$$\sum_{x\in\Z}\theta_x\ge c_2\min\set{\frac{n}{k_s},\frac{\sqrt{n}}{\sqrt{\log\log n}}l_s}\ge\frac{c'\sqrt{n}}{\sqrt{\log\log n}}\max_{x\in\Z}\theta_x.$$
We have $\Pr\left(\bigcup_{m=1}^{\infty}A_m\right)=1$. Also, let $B_n$ denote the event that for all~$s$ with $k_s\leq\frac{d_2\sqrt{n}}{2\sqrt{\log\log n}}$,
$$\sum_{x\in\Z}\left|f_{n}(x)\right|_{\Gamma_{s}}<\frac{\sigma}{4}\sum_{x\in\Z}\theta_x.$$

Conditioning on $A_m$, \eqref{eq:metric-lower-bound-eq1} shows that
$$\Pr\left(\sum_{x\in\Z}\left|f_{n}(x)\right|_{\Gamma_{s}}<\frac{\sigma}{4}\sum_{x\in\Z}\theta_x\suchthat A_m\right)\leq \exp\left(-\frac{c'\sqrt{n}}{\sqrt{\log\log n}}\right)$$
for all $n\ge m$. As $\set{k_s}$ increases exponentially, we may take a union bound to see
$$\Pr\left(B_n\suchthat A_m\right)\leq c''\log n\exp\left(-\frac{c'\sqrt{n}}{\sqrt{\log\log n}}\right)$$
for all $n\ge m$ and a constant $c''$. The latter expression is summable, so the Borel Cantelli Lemma shows that for each $m$, conditioning on~$A_m$ we have that $\Pr\left(\liminf B_n^c\suchthat A_m\right)=1$. But as $\Pr\left(\bigcup_{m=1}^{\infty}A_m\right)=1$, we must have $\Pr\left(\liminf B_n^c\right)=1$, as required.
\end{proof}

\section{Bounds on the distance in a single layer}\label{sec:bounds-sing-layer}

We use the results obtaines in the previous section to bound the distance of the random walk in each layer separately. As explained before, for a given $s$, the number of steps that the random walk on $\Delta_s=\Gamma_s\wr\Z$ makes in the copy of the lamp group $\Gamma_s$ at any given point $x\in\Z$  until time $n$ is $T(k_s,x,n)$, and the distance in each copy of $\Gamma_s$ is also bounded from above by $l_s$.

When $s$ is small, we expect to reach saturation in $\Gamma_s$ relatively fast, so the effective bound should be $l_s$. However, when $s$ is big, the number of excursions can be small compared to $l_s$, so the bound should depend on $T(k_s,x,n)$. We are therefore looking for the layers which will separate these two cases.

To do this, let $d_1,d_2,d_3$ be the constants from \Pref{prop:bound-Tkn-as}, \Pref{prop:chop-sum-as-bound} and \Cref{cor:chop-sum}. Take a constant $r\leq\frac{1}{2}\min\set{d_1,d_2,d_3}$, and consider the following layers:
\begin{align*}
s_0'(n)&=\max\set{s\ge 0\suchthat k_s\leq\frac{r\sqrt{n}}{\sqrt{\log\log n}}},\\
s_2(n)&=\max\set{s\ge 0\suchthat k_sl_s\leq\frac{r\sqrt{n}}{\sqrt{\log\log n}}},\\
s_3(n)&=\max\set{s\ge 0\suchthat k_sl_s\leq\sqrt{n\log\log n}},\\
\tilde{s}_3(n)&=\min\set{s_0'(n),s_3(n)}.
\end{align*}
We will prove that the dominating layer for the $\limsup$ is $s_2(n)$, whereas the dominating layer for the $\liminf$ is $\tilde{s}_3(n)$.

\subsection{Upper bounds}

\begin{prop}\label{prop:limsup-layer-upper-bound}
There is a universal constants $C>0$ such that almost surely, for large enough $n$ and all $s\ge 0$,
$$\left|W_n\right|_{\Delta_s}\leq\begin{cases}C\sqrt{n\log\log n}\, l_s,&s\leq s_2(n)\\ \frac{Cn}{k_s},&s_2(n)<s\leq s_0'(n)\\ C\sqrt{n\log\log n},&s>s_0'(n)\end{cases}.$$
\end{prop}

\begin{proof}
By \Pref{prop:up-dist-bound}, we have
$$|W_n|_{\Delta_s}\leq 11\min\set{T\left(\frac{k_s}{2},n\right)+\left|\range(S_n)\right|,\left|\range(S_n)\right|\, l_s}$$
for all $s\ge 0$. Also, by the classical Law of Iterated Logarithm, for large enough $n$ we have
\begin{equation}\label{eq:limsup-cl-LIL}
\left|\range(S_n)\right|\leq 2\sqrt{n\log\log n}.
\end{equation}
The case $s\leq s_2(n)$ follows from $|W_n|_{\Delta_s}\leq 11\cdot \left|\range(S_n)\right|\, l_s$ and \eqref{eq:limsup-cl-LIL}.

We move to the case where $s_2(n)<s\leq s_0'(n)$. By \Pref{prop:bound-Tkn-as}, almost surely for large enough $n$ and for all $s\leq s'_0(n)$,
$$T\left(\frac{k_s}{2},n\right)\leq\frac{C_1n}{k_s/2}=\frac{2C_1n}{k_s}.$$
Also, by \eqref{eq:limsup-cl-LIL}, for large enough $n$ we have
$$\left|\range(S_n)\right|\leq 2\sqrt{n\log\log n}\leq\frac{2rn}{k_s},$$
proving the second case.

For the case where $s>s'_0(n)$, as $k_s>\frac{r\sqrt{n}}{\sqrt{\log\log n}}$ we may apply \Lref{lem:ex-monotone} and \Pref{prop:bound-Tkn-as}:
$$T(k_s,n)\leq 3T\left(\frac{r\sqrt{n}}{2\sqrt{\log\log n}},n\right)\leq\frac{3C_1n}{\frac{r\sqrt{n}}{2\sqrt{\log\log n}}}=\frac{6C_1}{r}\sqrt{n\log\log n}.$$
Now the bound on $|W_n|_{\Delta_s}$ follows from \eqref{eq:limsup-cl-LIL}.
\end{proof}

\begin{prop}\label{prop:liminf-layer-upper-bound}
There is a universal constant $C>0$ such that almost surely, for all large enough $n$ for which $\left|\range(S_n)\right|\leq\frac{2\sqrt{n}}{\sqrt{\log\log n}}$ and all $s\ge 0$,
$$\left|W_n\right|_{\Delta_s}\leq\begin{cases}\frac{C\sqrt{n}}{\sqrt{\log\log n}}\, l_s,&s\leq \tilde{s}_3(n)\\ \frac{Cn}{k_s},&\tilde{s}_3(n)<s\leq s_0(n)\end{cases}.$$
\end{prop}

\begin{proof}
By \Pref{prop:up-dist-bound}, we have
$$|W_n|_{\Delta_s}\leq 11\min\set{T\left(\frac{k_s}{2},n\right)+\left|\range(S_n)\right|,\left|\range(S_n)\right|\, l_s}$$
for all $s\ge 0$. Also, by our assumption on $n$, $\left|\range(S_n)\right|\leq\frac{2\sqrt{n}}{\sqrt{\log\log n}}$, so the case $s\leq\tilde{s}_3(n)$ follows automatically.

We move to the case where $\tilde{s}_3(n)<s\leq s_0'(n)$. By \Pref{prop:bound-Tkn-as}, almost surely for large enough $n$ and for all $s\leq s'_0(n)$,
$$T\left(\frac{k_s}{2},n\right)\leq\frac{C_1n}{k_s/2}=\frac{2C_1n}{k_s}.$$
In addition,
$$\left|\range(S_n)\right|\leq \frac{2\sqrt{n}}{\sqrt{\log\log n}}\leq\frac{2rn}{k_s},$$
proving part of the second case.

For the case where $s'_0(n)<s\leq s_0(n)$, as $\frac{r\sqrt{n}}{\sqrt{\log\log n}}<k_s\leq\frac{2\sqrt{n}}{\sqrt{\log\log n}}$ we may apply \Lref{lem:ex-monotone} and \Pref{prop:bound-Tkn-as}:
\begin{align*}
T(k_s,n)&\leq 3T\left(\frac{r\sqrt{n}}{2\sqrt{\log\log n}},n\right)\leq\frac{C_1n}{\frac{r\sqrt{n}}{2\sqrt{\log\log n}}}=\\
&=\frac{2C_1}{r}\sqrt{n\log\log n}<\frac{2rn}{k_s}.
\end{align*}
The conclusion follows.
\end{proof}

\subsection{Lower bounds}

\begin{prop}\label{prop:liminf-layer-lower-bound}
There is a universal constant $c>0$ such that almost surely, for all but finitely many $n$,
$$\left|W_n\right|_{\Delta_s}\ge\begin{cases}\frac{c\sqrt{n}}{\sqrt{\log\log n}}l_s,&s\leq \tilde{s}_3(n)\\\frac{cn}{k_s},&\tilde{s}_3(n)<s\leq s_0'(n).\end{cases}$$
\end{prop}
\begin{proof}
Using \Pref{prop:metric-lower-bound}, for large enough $n$ we have
$$\left|W_n\right|_{\Delta_s}\ge c\sum_{x\in\Z}\min\set{T(k_s,x,n),l_s}$$
for all $s\leq s_0'(n)$. The result follows from \Pref{prop:chop-sum-as-bound}.
\end{proof}

\begin{prop}\label{prop:limsup-layer-lower-bound}
There is a universal constant $c>0$ such that almost surely, for infinitely many $n$,
$$\left|W_n\right|_{\Delta_s}\ge\begin{cases}c\sqrt{n\log\log n}\, l_s,&s\leq s_2(n)\\\frac{cn}{k_s},&s_2(n)<s\leq s_0'(n).\end{cases}$$
\end{prop}
\begin{proof}
Using \Pref{prop:metric-lower-bound}, for large enough $n$ we have
$$\left|W_n\right|_{\Delta_s}\ge c\sum_{x\in\Z}\min\set{T(k_s,x,n),l_s}$$
for all $s\leq s_0'(n)$. The result follows from \Cref{cor:chop-sum}.
\end{proof}

\section{Bounds on the total distance}\label{sec:bounds-total}

We now use the bounds on each layer achieved in the previous section to find a bound on the distance of the random walk on $\Delta$ in terms of the sequences $\set{k_s}$ and $\set{l_s}$. Recall the following bounds from \cite{brzh}:
\begin{equation}\label{eq:total-dist-up-bound}
    \left|W_n\right|_{\Delta}\leq C\sum_{s\leq s_0(n)}\left|W_n\right|_{\Delta_s}
\end{equation}
where $s_0(n)=\min\set{s\ge 0\suchthat k_s\leq\left|\range(S_n)\right|}$, and
\begin{equation}\label{eq:total-dist-low-bound}
    \left|W_n\right|_{\Delta}\ge\left|W_n\right|_{\Delta_s}
\end{equation}
for any $s\ge 0$.

For the $\limsup$, recall the critical layer
$$s_2(n)=\max\set{s\ge 0\suchthat k_sl_s\leq\frac{r\sqrt{n}}{\sqrt{\log\log n}}}$$
and that
$$s'_0(n)=\max\set{s\ge 0\suchthat k_s\leq\frac{r\sqrt{n}}{\sqrt{\log\log n}}}$$
where $r>0$ is some universal constant.

\begin{prop}\label{prop:limsup-s-expr}
Suppose $\set{\Gamma_s}$ satisfy the $(\sigma,c_0l_s)$-linear speed assumption and $\diam(\Gamma_s)\leq C_0l_s$. Suppose there exists a constant $m_0>1$ such that
$$k_{s+1}>2k_s,\; l_{s+1}\ge m_0l_s$$
for all $s$. Let
$$g(n)=\frac{n}{k_{s_2(n)+1}}+\sqrt{n\log\log n}l_{s_2(n)}.$$
Then almost surely,
$$0<\limsup_{n\to\infty}\frac{\left|W_n\right|_{\Delta}}{g(n)}\;\;\textrm{and}\;\;\limsup_{n\to\infty}\frac{\left|W_n\right|_{\Delta}}{g(n)+\sqrt{n\log\log n}\log\log\log n}<\infty.$$
\end{prop}
\begin{proof}
For the lower bound, we use \eqref{eq:total-dist-low-bound} and \Pref{prop:limsup-layer-lower-bound}, showing that for infinitely many $n$ we have
\begin{align*}
    \left|W_n\right|_{\Delta}&\ge\frac{1}{2}\left(\left|W_n\right|_{\Delta_{s_2(n)}}+\left|W_n\right|_{\Delta_{s_2(n)+1}}\right)\ge\\
    &\ge\frac{c}{2}\left(\sqrt{n\log\log n}l_{s_2(n)}+\frac{n}{k_{s_2(n)+1}}\right)=\frac{c}{2}g(n),
\end{align*}
so $\limsup_{n\to\infty}\frac{\left|W_n\right|_{\Delta}}{g(n)}>0$ almost surely.

For the upper bound, we use \eqref{eq:total-dist-up-bound}. We divide the interval $0\leq s\leq s_0(n)$ into three parts, and use \Pref{prop:limsup-layer-upper-bound} to analyze the contribution of each of them:
\begin{itemize}
  \item For $0\leq s\leq s_2(n)$,
  \begin{align*}
  \sum_{s\leq s_2(n)}|W_n|_{\Delta_s}&\leq\sum_{s\leq s_2(n)}C_1\sqrt{n\log\log n}\, l_s\leq\\
  &\leq\frac{C_1}{1-1/m_0}\sqrt{n\log\log n}\, l_{s_2(n)}
  \end{align*}
  where the last inequality follows from $l_{s+1}\ge m_0l_s$.
  \item For $s_2(n)+1\leq s\leq s_0'(n)$,
  $$\sum_{s_2(n)+1\leq s\leq s_0'(n)}|W_n|_{\Delta_s}\leq\sum_{s_2(n)+1\leq s\leq s_0'(n)}\frac{C_1n}{k_s}\leq\frac{2C_1n}{k_{s_2(n)+1}}$$
  where the last inequality follows from $k_{s+1}>2k_s$.
  \item For $s_0'(n)+1\leq s\leq s_0(n)$, recall that for large enough $n$ we have $\left|\range(S_n)\right|\leq 2\sqrt{n\log\log n}$. As $k_{s+1}>2k_s$, the number of layers $s_0'(n)+1\leq s\leq s_0(n)$ is at most $c\log\log\log n$ for some constant $c>0$, so
      $$\sum_{s_0'(n)+1\leq s\leq s_0(n)}|W_n|_{\Delta_s}\leq C_1\sqrt{n\log\log n}\log\log\log n.$$
\end{itemize}
Summing the above yields
$$\left|W_n\right|_{\Delta}\leq Cg(n)+C\sqrt{n\log\log n}\log\log\log n,$$
proving the second $\limsup$ is finite.
\end{proof}

For the liminf, recall that
$$s_3(n)=\max\set{s\ge 0\suchthat k_sl_s\leq\sqrt{n\log\log n}}$$
and that $\tilde{s}_3(n)=\min\set{s'_0(n),s_3(n)}$.
\begin{prop}\label{prop:liminf-s-expr}
Suppose $\set{\Gamma_s}$ satisfy the $(\sigma,c_0l_s)$-linear speed assumption and $\diam(\Gamma_s)\leq C_0l_s$. Suppose there exists a constant $m_0>1$ such that
$$k_{s+1}>2k_s,\; l_{s+1}\ge m_0l_s$$
for all $s$. Let
$$h(n)=\begin{cases}\frac{n}{k_{s_3(n)+1}}+\frac{\sqrt{n}}{\sqrt{\log\log n}}l_{s_3(n)},&s_3(n)<s'_0(n)\\\frac{\sqrt{n}}{\sqrt{\log\log n}}l_{s'_0(n)},&\textrm{otherwise}.\end{cases}$$
Then almost surely
$$0<\liminf_{n\to\infty}\frac{|W_n|_{\Delta}}{h(n)}<\infty.$$
\end{prop}
\begin{proof}
For the upper bound we use again \eqref{eq:total-dist-up-bound}. Take $n$ such that $\left|\range(S_n)\right|\leq\frac{2\sqrt{n}}{\sqrt{\log\log n}}$, and assume $n$ is large enough such that \Pref{prop:liminf-layer-upper-bound} holds. We divide the interval $0\leq s\leq s_0(n)$ into several parts and analyze the contribution of each of them:
\begin{itemize}
  \item For $0\leq s\leq\tilde{s}_3(n)$,
  $$\sum_{s\leq \tilde{s}_3(n)}|W_n|_{\Delta_s}\leq\sum_{s\leq \tilde{s}_3(n)}\frac{C\sqrt{n}}{\sqrt{\log\log n}}\, l_s\leq\frac{C}{1-1/m_0}\frac{\sqrt{n}}{\sqrt{\log\log n}}l_{\tilde{s}_3(n)}$$
  where the last inequality follows from $l_{s+1}\ge m_0l_s$.
  \item If $s_3(n)<s'_0(n)$, then for $s_3(n)+1\leq s\leq s'_0(n)$ we have
  $$\sum_{s_3(n)+1\leq s\leq s'_0(n)}|W_n|_{\Delta_s}\leq\sum_{s_3(n)+1\leq s\leq s'_0(n)}\frac{Cn}{k_s}\leq\frac{2Cn}{k_{s_3(n)+1}}$$
  where the last inequality follows from $k_{s+1}>2k_s$.
  \item For $s'_0(n)+1\leq s\leq s_0(n)$, we have
  \begin{align*}
      |W_n|_{\Delta_s}&\leq T(k_s,n)+\left|\range(S_n)\right|\leq \\
      &\leq 3T(k_{s'_0(n)},n)+|\range(S_n)|\leq\frac{Cn}{k_{s'_0(n)}}.
  \end{align*}
  Since $\set{k_s}$ grows at least exponentially, and by our assumption on $n$ $\left|\range(S_n)\right|\leq\frac{2\sqrt{n}}{\sqrt{\log\log n}}$ the number of such layers is bounded above by a constant, so
  $$\sum_{s'_0(n)+1\leq s\leq s_0(n)}\left|W_n\right|_{\Delta_s}\leq\frac{Cn}{k_{s'_0(n)}}.$$
\end{itemize}
Summing the above, we have
$$\left|W_n\right|_{\Delta}\leq\begin{cases}\frac{C\sqrt{n}}{\sqrt{\log\log n}}l_{s_3(n)}+\frac{Cn}{k_{s_3(n)+1}}+\frac{Cn}{k_{s'_0(n)}},&s_3(n)<s'_0(n)\\\frac{C\sqrt{n}}{\sqrt{\log\log n}}l_{s'_0(n)}+\frac{Cn}{k_{s'_0(n)}},&\textrm{otherwise}.\end{cases}$$
In the first case, $k_{s'_0(n)}\ge k_{s_3(n)+1}$, so the latter term is inessential; in the second case, $s'_0(n)\leq s_3(n)$, so $k_{s'_0(n)}l_{s'_0(n)}\leq\sqrt{n\log\log n}$, and again the latter term is inessential. We therefore get $\left|W_n\right|_{\Delta}\leq 2Ch(n)$ for each such $n$, proving the the above $\liminf$ is finite.

To prove that the $\liminf$ is positive, we use \eqref{eq:total-dist-low-bound}. Take $n$ large enough so that \Pref{prop:liminf-layer-lower-bound} holds. If $s_3(n)<s'_0(n)$, we have
\begin{align*}
    \left|W_n\right|_{\Delta}&\ge\frac{1}{2}\left(\left|W_n\right|_{\Delta_{s_3(n)}}+\left|W_n\right|_{\Delta_{s_3(n)+1}}\right)\ge\\
    &\ge\frac{c}{2}\left(\frac{\sqrt{n}}{\sqrt{\log\log n}}l_{s_3(n)}+\frac{n}{k_{s_3(n)+1}}\right)=\frac{c}{2}h(n).
\end{align*}
Finally, if $s'_0(n)\leq s_3(n)$, we have
$$\left|W_n\right|_{\Delta}\ge\left|W_n\right|_{\Delta_{s'_0(n)}}\ge\frac{c\sqrt{n}}{\sqrt{\log\log n}}l_{s'_0(n)}=c\,h(n).$$
This concludes the proof.
\end{proof}

\section{Proof of the main theorem}\label{sec:main-thm-functions}

We return to the idea of \cite{brzh} for approximating a given function. Let $f:\N\to\N$ be a function such that $\frac{f(n)}{\sqrt{n}}$ and $\frac{n}{f(n)}$ are non-decreasing. Choose sequences $\set{k_s}$ and $\set{l_s}$ as in \Pref{prop:approx-funcs}, and let $\Delta$ be the corresponding diagonal product. Denote by $W_n$ the random walk on $\Delta$ with respect to the ``switch-walk-switch'' measure.

\begin{thm}
If $\frac{f(n)}{\sqrt{n}\log\log\log n}$ is non-decreasing, then almost surely
$$0<\limsup_{n\to\infty}\frac{|W_n|_{\Delta}}{\log\log n\cdot f(\frac{n}{\log\log n})}<\infty.$$
\end{thm}
\begin{proof}
We first note that, by \Pref{prop:approx-funcs},
\begin{align*}
  f\left(\frac{rn}{\log\log n}\right) & \simeq\sqrt{\frac{rn}{\log\log n}}\,l_{s_2(n)}+\frac{\frac{rn}{\log\log n}}{k_{s_2(n)+1}}\simeq\\
  &\simeq\frac{1}{\log\log n}\left(\sqrt{n\log\log n}\,l_{s_2(n)}+\frac{n}{k_{s_2(n)+1}}\right).
\end{align*}
Therefore, \Pref{prop:limsup-s-expr} shows that almost surely,
$$0<\limsup_{n\to\infty}\frac{\left|W_n\right|_{\Delta}}{\log\log n\left(f\left(\frac{n}{\log\log n}\right)+\frac{\sqrt{n}}{\sqrt{\log\log n}}\log\log\log n\right)}<\infty.$$
As $\frac{f(n)}{\sqrt{n}\log\log\log n}$ is non-decreasing,
$$\frac{\sqrt{n}}{\sqrt{\log\log n}}\log\log\log n\lesssim f\left(\frac{n}{\log\log n}\right),$$
and the conclusion follows.
\end{proof}

\begin{thm}
If $\frac{f(n)}{\sqrt{n}(\log\log n)^{1+\eps}}$ is non-decreasing for some $\eps>0$, and the sequences $\set{k_s}$ and $\set{l_s}$ are chosen so that $\log\log k_s\leq l_s$ (as in \Lref{lem:approx-loglog}), then almost surely
$$0<\liminf_{n\to\infty}\frac{|W_n|_{\Delta}}{\frac{1}{\log\log n}\cdot f(n\log\log n)}<\infty.$$
\end{thm}
\begin{proof}
We first note that, by \Pref{prop:approx-funcs},
\begin{align*}
  f\left(c_1n\log\log n\right) & \simeq\sqrt{c_1n\log\log n}\, l_{s_3(n)}+\frac{c_1n\log\log n}{k_{s_3(n)+1}}\simeq\\
  &\simeq\log\log n\left(\frac{\sqrt{n}}{\sqrt{\log\log n}}\,l_{s_3(n)}+\frac{n}{k_{s_3(n)+1}}\right).
\end{align*}
Also, the assumption that $\log\log k_s\leq l_s$ is non-decreasing shows that $s_3(n)\leq s_0'(n)$ for large enough $n$. Therefore, \Pref{prop:liminf-s-expr} shows that almost surely
$$0<\liminf_{n\to\infty}\frac{\left|W_n\right|_{\Delta}}{\frac{1}{\log\log n}f(n\log\log n)}<\infty.$$
\end{proof}

\appendix

\section{$k$-jumps of a simple random walk on $\Z$}\label{app:random-walk-facts}

Let $S_t$ be a simple random walk on $\Z$, and fix $k>0$. Recall from \Dref{defn:k-assoc} the sequence $\set{n_j^{(k)}}_{j=0}^{\infty}$: We set $n_0^{(k)}=0$, and
$$n_j^{(k)}=\inf\set{t>n_{j-1}^{(k)}\suchthat S_t\in\set{S_{n_{j-1}^{(k)}}\pm k}}.$$
Recall also $N_k(n)=\max\set{j\suchthat n_j^{(k)}\leq n}$. Let $X_j=n_j^{(k)}-n_{j-1}^{(k)}$. Our goal is to estimate $N_k(n)$.

\begin{prop}
Let $0<r<1$. There exists $c>0$ such that for any $n,k\ge 1$,
$$\Pr\left(N_k(n)>\frac{cn}{k^2}\right)\leq \exp\left(-\frac{rn}{k^2}\right).$$
\end{prop}
\begin{proof}
For $k=1$, we may choose $c=1$ and the inequality will hold; so we assume $k>1$. Note that
\begin{align*}
\Pr\left(X_j>k^2\right)&=\Pr\left(\max_{n_{j-1}\leq t\leq n_{j-1}+k^2}\left|S_t-S_{n_{j-1}}\right|<k\right)=\\
&=\Pr\left(\max_{0\leq t\leq k^2}\left|S_t\right|<k\right)=1-\Pr\left(\max_{0\leq t\leq k^2}\left|S_t\right|\ge k\right)\ge\\
&\ge 1-\frac{\E|S_{k^2}|}{k},
\end{align*}
where the last step follows from Doob's maximal inequality. Since $\E|S_n|\sim\sqrt{\frac{2}{\pi}n}$, there is a constant $p_0>0$ such that $\Pr\left(X_j>k^2\right)\ge p_0$ for all $k$. As the random variables $X_1,X_2,\dots$ are i.i.d., for all $c>1$ we have
\begin{align*}
  \Pr\left(N_k(n)>\frac{cn}{p_0k^2}\right) & =\Pr\left(\sum_{j=1}^{cn/p_0k^2}X_j\leq n\right)\leq\\
  &\leq\Pr\left(\sum_{j=1}^{cn/p_0k^2}k^2\cdot \ind{\set{X_j>k^2}}\leq n\right)\leq\\
  &\leq\Pr\left(\Bin\left(\frac{cn}{p_0k^2},p_0\right)\leq\frac{n}{k^2}\right).
\end{align*}
Chernoff's inequality shows that
$$\Pr\left(\Bin\left(\frac{cn}{p_0k^2},p_0\right)\leq\frac{n}{k^2}\right)\leq\exp\left(-\left(1-\frac{1}{c}\right)^2\frac{n}{2k^2}\right).$$

Now, let $0<r<1$. Choosing a large enough $c$ so that $r\leq\left(1-\frac{1}{c}\right)^2$, we have
$$\Pr\left(N_k(n)>\frac{cn}{p_0k^2}\right)\leq\exp\left(-\frac{rn}{k^2}\right)$$
as required.
\end{proof}

\begin{prop}
Let $0<r<\frac{1}{10}$. Then there exists $c'>0$ such that for all $n,k\ge 1$,
$$\Pr\left(N_k(n)<\frac{c'n}{k^2}\right)\leq\exp\left(-\frac{rn}{k^2}\right).$$
\end{prop}
\begin{proof}
First, for $k=1$ we may take $c'=\frac{1}{2}$, and the desired probability is $0$; so we assume $k>1$. Note that
$$\Pr\left(X_i\ge k^2\right)\leq\Pr\left(\left|S_{k^2}\right|\leq k\right),$$
and the latter converges from below to $\Pr\left(-1\leq \mathrm{N}(0,1)\leq 1\right)\approx 0.683$. One can check that by taking $p_1=0.9$ we have $\Pr\left(X_i\ge k^2\right)\leq p_1$ for all $k\ge 2$.

Now, for any $m\ge 1$,
\begin{align*}
  \Pr\left(X_j\ge(m+1)k^2\right)&=\Pr\left(X_j\ge(m+1)k^2\suchthat X_j\ge k^2\right)\cdot\Pr(X_j\ge k^2)\leq\\
  &\leq\Pr\left(X_j\ge mk^2\right)\Pr\left(X_j\ge k^2\right),
\end{align*}
and induction on $m$ shows that $\Pr(X_j\ge mk^2)\leq\Pr\left(X_j\ge k^2\right)^m$ for all~$m$. Therefore random variable $\frac{1}{k^2}X_j$ is stochastically dominated by a geometric random variable $G_j\sim\Geo(1-p_1)$, so for any $c'\leq 1-p_1$,
$$\Pr\left(N_k(n)<\frac{c'n}{k^2}\right)=\Pr\left(\sum_{j=1}^{c'n/k^2}X_j>n\right)\leq \Pr\left(\sum_{j=1}^{c'n/k^2}G_j>\frac{n}{k^2}\right).$$
By \cite[Theorem 2.1]{jan},
$$\Pr\left(\sum_{j=1}^{c'n/k^2}G_j>\frac{n}{k^2}\right)\leq\exp\left(-\frac{c'n}{k^2}(\lambda-1-\ln\lambda )\right)$$
for $\lambda=\frac{1-p_1}{c'}$.

Finally, let $0<r<\frac{1}{10}$. As $c'(\lambda-1-\ln\lambda)=1-p_1-c'-c'\ln\frac{1-p_1}{c'}\to 1-p_1$ as $c'\to 0$, by choosing small enough $c'$ we have $c'(\lambda-1-\ln\lambda)>r$, and thus
$$\Pr\left(N_k(n)>\frac{c'n}{k^2}\right)\leq\exp\left(-\frac{rn}{k^2}\right).$$
\end{proof}

\section{On the maximal local time of a simple random walk}\label{app:max-local-time-brow}

For a simple random walk $S_t$ on $\Z$, recall that
$$L(x,n)=\left|\set{0\leq k\leq n\suchthat S_k=x}\right|$$
denotes the local time of $S_t$ at $x$ until time $n$. We are interested in the tail behavior of $L(n)=\max_{x\in\Z} L(x,n)$.
The main aim of this section is to prove an upper tail bound on the maximal local time of a simple random walk on $\Z$ (\Cref{cor:max-local-time}). This seems like a classical result, however we could not find the exact statement we needed in the literature, and therefore provide the statement with proof and background for completeness.

\subsection{Brownian motion and the Skorokhod embedding}

Let $B_t$ be a Brownian motion on $\R$ starting at $0$. Define a sequence of stopping times by $\tau_0=0$ and, inductively,
$$\tau_k=\inf\set{t>\tau_{k-1}\suchthat \left|B_t-B_{\tau_{k-1}}\right|=1}.$$
Then $S_k=B_{\tau_k}$ is a simple random walk on $\Z$, and $\tau_k-\tau_{k-1}$ are i.i.d.r.v.\ with $\E\left[\tau_k-\tau_{k-1}\right]=1$ and $\sigma^2=\E\left[\left(\tau_k-\tau_{k-1}\right)^2\right]<\infty$. We therefore have
\begin{equation}\label{eq:tau_n-n}
\Pr\left(\left|\tau_n-n\right|\ge\sqrt{n}u\right)\leq 2\exp\left(-\frac{u^2}{2\sigma^2}\right).
\end{equation}

\subsection{Brownian local time}

Given a Brownian motion $B_t$ on $\R$, a theorem of Trotter \cite{trot} shows that almost surely there exists a function $\eta(x,t)$, jointly continuous in $x$ and $t$, such that
$$\eta(x,t)=\frac{\dd\,}{\dd x\,}\int_0^t \ind{(-\infty,x)}(B_s)\dd s.$$
$\eta(x,t)$ is called the \textbf{local time of $B_t$ at $x$}. The distribution of $\eta(0,t)$ is well known: for all $u>0$,
\begin{equation}\label{eq:eta-dist}
\Pr\left(\frac{\eta(0,t)}{\sqrt{t}}\leq u\right)=2\Phi(u)-1
\end{equation}
(see e.g.\ \cite{kes}).

Let $\eta(t)=\sup_{x\in\R}\eta(x,t)$. By the scaling property of Brownian motion, $\frac{\eta(t)}{\sqrt{t}}$ has the same distribution as $\eta(1)$, which, by \cite{kes}, satisfies
\begin{equation}\label{eq:eta-t-dist}
\Pr\left(\eta(t)\ge\sqrt{t}u\right)=\Pr\left(\eta(1)\ge u\right)\leq\exp\left(-\frac{u^2}{4}\right)
\end{equation}
for large enough $u$.

\subsection{Strong approximation for the local time}

We give a strong approximation theorem for the maximal local time of a random walk and a Brownian motion, using the Skorokhod embedding. This is similar to other results (such as in \cite{revesz, csre, japr2}), but with a more quantitative flavor.

We use the notations of \cite{revesz}. Fix $x\in\Z$. Define
$$\nu_1=\inf\set{k\ge 0\suchthat B_{\tau_k}=S_k=x}$$
and, inductively,
$$\nu_j=\inf\set{k>\nu_{j-1}\suchthat B_{\tau_k}=S_k=x}.$$
For each $j$, let
$$a_j(x)=\eta(x,\tau_{\nu_j+1})-\eta(x,\tau_{\nu_j-1}).$$
From \cite{revesz}, $a_1(x),a_2(x),\dots$ are i.i.d.r.v.\ with
\begin{equation}\label{eq:aj-dist}
\Pr\left(a_j(x)\ge u\right)=\Pr\left(\eta(0,\tau_1)\ge u\right)\leq C\exp\left(-\frac{\pi}{4}u\right)
\end{equation}
(for some universal constant $C>0$), and $\E[a_j(x)]=1$. In addition, the random matrices $L=\left(L(x,n)\right)_{x\in\Z,n\ge 0}$ and $A=\left(a_j(x)\right)_{x\in\Z,j\ge 1}$ are independent.

\begin{lem}\label{lem:br-approx-1}
There is a universal constant $c_1>0$ such that for any $x\in\Z$ and $u\ge 1$,
$$\Pr\left(\left|a_1(x)+\cdots+a_{L(x,n)}(x)-L(x,n)\right|\ge\sqrt{n}u\right)\leq c_1\exp\left(-\frac{n^{1/3}u^{4/3}}{2}\right)$$
\end{lem}
\begin{proof}
By the above, we have
\begin{align*}
  &\Pr\left(\left|a_1(x)+\cdots+a_{L(x,n)}(x)-L(x,n)\right|\ge\sqrt{n}u\right)=\\
  &=\sum_{k=0}^n\Pr\left(\left|a_1(x)+\cdots+a_{k}(x)-k\right|\ge\sqrt{n}u\right)\Pr\left(L(x,n)=k\right)=\\
  &=\sum_{k=0}^n\Pr\left(\left|\frac{a_1(x)+\cdots+a_{k}(x)-k}{\sqrt{k}}\right|\ge\frac{\sqrt{n}u}{\sqrt{k}}\right)\Pr\left(L(x,n)=k\right)\leq\\
  &\leq\sum_{k=0}^{(nu)^{2/3}}\Pr\left(\left|\frac{a_1(x)+\cdots+a_{k}(x)-k}{\sqrt{k}}\right|\ge\frac{\sqrt{n}u}{\sqrt{k}}\right)\Pr\left(L(x,n)=k\right)+\\
  &+\Pr\left(L(x,n)> (nu)^{2/3}\right)\leq\\
  &\leq\sum_{k=0}^{(nu)^{2/3}}\Pr\left(\left|\frac{a_1(x)+\cdots+a_{k}(x)-k}{\sqrt{k}}\right|\ge n^{1/6}u^{2/3}\right)\Pr\left(L(x,n)=k\right)+\\
  &+C\exp\left(-\frac{n^{1/3}u^{4/3}}{2}\right)\leq\\
  &\leq 2\exp\left(-\frac{n^{1/3}u^{4/3}}{2}\right)+C\exp\left(-\frac{n^{1/3}u^{4/3}}{2}\right)\leq \tilde{C}\exp\left(-\frac{n^{1/3}u^{4/3}}{2}\right).
\end{align*}
\end{proof}

\begin{lem}\label{lem:br-approx-2}
There is a universal constant $C_1>0$ such that for any $x\in\Z$ and $u\ge 1$,
$$\Pr\left(\left|a_1(x)+\cdots+a_{L(x,n)}(x)-\eta(x,\tau_n)\right|\ge\sqrt{n}u\right)\leq C_1\exp\left(-\frac{\pi\sqrt{n}u}{4}\right)$$
\end{lem}
\begin{proof}
By \cite[equation (2.6)]{revesz},
$$\left|a_1(x)+\cdots+a_{L(x,n)}(x)-\eta(x,\tau_n)\right|\leq \gamma+a_{L(x,n)}(x)$$
for some constant $\gamma\ge 0$. Therefore, by \eqref{eq:aj-dist},
\begin{align*}
&\Pr\left(\left|a_1(x)+\cdots+a_{L(x,n)}(x)-\eta(x,\tau_n)\right|\ge\sqrt{n}u\right)\leq\\
&\leq\Pr\left(\gamma+a_{L(x,n)}(x)\ge\sqrt{n}u\right)\leq C\exp\left(-\frac{\pi}{4}(\sqrt{n}u-\gamma)\right)
\end{align*}
as required.
\end{proof}

\begin{cor}\label{cor:max-local-time}
There are universal constants $C,C'>0$ such that for all $n,u\ge 1$,
$$\Pr\left(L(n)\ge\sqrt{n}u\right)\leq Cu^2\exp\left(-C'u^2\right).$$
\end{cor}
\begin{proof}
If $u\ge\sqrt{n}$, the claim is trivial. So we assume $u<\sqrt{n}$. By combining \Lref{lem:br-approx-1} and \Lref{lem:br-approx-2}, for any $x\in\Z$ we have
\begin{align*}
&\Pr\left(\left|L(x,n)-\eta(x,\tau_n)\right|\ge 2\sqrt{n}u\right)\leq\\
&\leq c_1\exp\left(-\frac{n^{1/3}u^{4/3}}{2}\right)+C_1\exp\left(-\frac{\pi\sqrt{n}u}{4}\right).
\end{align*}
Taking a union bound, we have
\begin{align*}
&\Pr\left(\sup_{\left|x\right|\leq n}\left|L(x,n)-\eta(x,\tau_n)\right|\ge 2\sqrt{n}u\right)\leq\\ &\leq Cn\left(c_1\exp\left(-\frac{n^{1/3}u^{4/3}}{2}\right)+C_1\exp\left(-\frac{\pi\sqrt{n}u}{4}\right)\right).
\end{align*}
The function $n\mapsto n\exp\left(n^{\alpha}u^{\beta}\right)$ for fixed $u,\alpha,\beta\ge 1$ is decreasing for $n\ge 1$, so we may use $n>u^2$ and get
$$\Pr\left(\sup_{\left|x\right|\leq n}\left|L(x,n)-\eta(x,\tau_n)\right|\ge 2\sqrt{n}u\right) \leq Cu^2\exp\left(-\frac{u^2}{2}\right)$$
for an appropriate constant $C>0$. The claim now follows from \eqref{eq:tau_n-n} and \eqref{eq:eta-t-dist}.
\end{proof}

\section{Approximation of functions}
In \cite[Appendix B]{brzh}, Brieussel and Zheng show how to approximate a function $f\colon [1,\infty)\to[1,\infty)$ such that $\frac{f(x)}{\sqrt{x}}$ and $\frac{x}{f(x)}$ are non-decreasing by a function of the form
$$\bar{f}(x)=\frac{x}{k_{s+1}}+\sqrt{x}l_s,\;\;\;(k_sl_s)^2\leq x<(k_{s+1}l_{s+1})^2$$
for appropriate sequences of nonnegative integers $\set{k_s}$ and $\set{l_s}$. Here we prove the following lemma:
\begin{lem}\label{lem:approx-loglog}
Let $f\colon[1,\infty)\to[1,\infty)$ be a continuous function such that $f(1)=1$ and $\frac{x}{f(x)}$ and $\frac{f(x)}{\sqrt{x}(\log\log x)^{1+\eps}}$ are non-decreasing for some $\eps>0$, and let $m_0>1$. Then one can find sequences $\set{k_s}$ and $\set{l_s}$, possibly finite with last value infinity in one of them, such that:
\begin{itemize}
    \item $k_{s+1}\ge m_0k_s$ and $l_{s+1}\ge m_0l_s$ for all $s$;
    \item $f(x)\simeq_{2m_0}\bar{f}(x)$;
    \item and $\log\log k_s\leq l_s$ for large enough $s$.
\end{itemize}
\end{lem}
\begin{proof}
Writing $g(x)=\frac{f(x^2)}{x}$, we have $\frac{g(x)}{(\log\log x)^{1+\eps}}$ and $\frac{x}{g(x)}$ are non-decreasing. We can now use \cite[Lemma B.1]{brzh} to find sequences $\set{k_s}$ and $\set{l_s}$ so that the first two conditions of the lemma are satisfied.

To prove the last assertion, we recall the construction of $\set{k_s}$ and~$\set{l_s}$. They are defined inductively, according to the following procedure: Define first $k_0=l_0=1$. Assuming that $k_s$ and $l_s$ were defined and that $g(k_sl_s)=l_s$, we have $l_s\leq g(x)\leq\frac{x}{k_s}$ for all $x\ge k_sl_s$. Take the minimal $y\ge m_0^2k_sl_s$ such that $m_0l_s\leq g(y)\leq\frac{y}{m_0k_s}$. If such $y$ exists, we have two cases:
\begin{enumerate}
    \item $g(y)=\frac{y}{k_s}$ -- in which case we take $k_{s+1}=m_0k_s$ and $l_{s+1}=\frac{y}{m_0k_s}\ge m_0l_s$;
    \item $g(y)=m_0l_s$ -- in which case we take $l_{s+1}=m_0l_s$ and $k_{s+1}=\frac{y}{m_0l_s}\ge m_0k_s$.
\end{enumerate}
If such $y$ does not exist, the assumption $\frac{g(x)}{(\log\log x)^{1+\eps}}$ shows that we have $g(x)\ge\frac{x}{m_0k_s}$ for all $x\ge k_sl_s$, in which case we take $k_{s+1}=m_0k_s$ and $l_{s+1}=\infty$.

To ensure the last condition, we only need to check that
\begin{equation}\label{eq:approx-ineq}
    g(m_0l_s\exp(\exp(m_0l_s)))\ge m_0l_s.
\end{equation}
Indeed, if \eqref{eq:approx-ineq} holds, then in the second case we have $y\leq m_0l_s\exp\exp(m_0l_s)$, so $k_{s+1}\leq\exp\exp(m_0l_s)=\exp\exp(l_{s+1})$ as required. To prove \eqref{eq:approx-ineq}, note that
\begin{align*}
    g(m_0l_s\exp(\exp(m_0l_s)))&\ge g(k_sl_s)\cdot\left(\frac{\log\log(m_0l_s\exp(\exp(m_0l_s)))}{\log\log(k_sl_s)}\right)^{1+\eps}\ge\\
    &\ge l_s\cdot\left(\frac{m_0l_s}{\log\log k_s+\log\log l_s}\right)^{1+\eps}\ge\\
    &\ge m_0^{1+\eps}l_s\left(\frac{l_s}{l_s+\log\log l_s}\right)^{1+\eps}
\end{align*}
and the last parenthesis tends to $1$ as $l_s$ tends to $\infty$.
\end{proof}

\section*{Acknowledgments}
The authors were supported by Israeli Science Foundation grant \#957/20. The second author was also supported by the Bar-Ilan President's Doctoral Fellowships of Excellence.

\bibliographystyle{plain}
\bibliography{refs}

\end{document}